\documentclass[11pt,oneside,reqno]{amsart}

\usepackage{amssymb}
\usepackage{color}
\usepackage{tikz-cd}
\usetikzlibrary{decorations.markings}

\usepackage{xcolor}
\definecolor{deepgreen}{cmyk}{0.99998,0,1,0}

\usepackage{hyperref}
\hypersetup{hypertex=true,
	colorlinks=true,
	linkcolor=blue,
	anchorcolor=blue,
	citecolor=blue}
\usepackage{comment}
\usepackage{mathrsfs}

\allowdisplaybreaks[4]

\usepackage{etoolbox}

\usepackage{bm}
\usepackage{bbm}

\usepackage[new]{old-arrows}

\theoremstyle{definition}
\newtheorem{defi}{$\mathbf{Definition}$}[section]
\newtheorem*{pro}{$\mathbf{Proof}$}

\theoremstyle{plain}
\newtheorem{theo}[defi]{$\mathbf{Theorem}$}
\newtheorem{lemma}[defi]{$\mathbf{Lemma}$}

\newtheorem{prop}[defi]{$\mathbf{Proposition}$}

\newcommand{\tro}{\mathrm{Tr}}

\newcommand{\pa}{\partial}

\newcommand{\lv}{\left\vert}
\newcommand{\rv}{\right\vert}
\newcommand{\lV}{\left\Vert}
\newcommand{\rV}{\right\Vert}

\newcommand{\bv}{\big\vert}
\newcommand{\Bv}{\Big\vert}
\newcommand{\bbv}{\bigg\vert}

\newcommand{\cref}{\S\ \ref}
\newcommand{\Cref}{\S\ \ref}

\definecolor{pink}{RGB}{249,164,186}
\definecolor{grassgreen}{RGB}{128,255,0}

\numberwithin{equation}{section}
\title{Arithmetic Uniform Ergodicity for Flat Vector Bundles}

\author{Qiaochu Ma$^{1,2}$}
\address{$^1$Department of Mathematics, The Pennsylvania State University}
\address{$^2$Department of Mathematics, Texas A\&M University}
\date{}

\makeatletter
\newcommand*{\rom}[1]{\expandafter\@slowromancap\romannumeral #1@}
\makeatother

\begin{document} 
	\clearpage\maketitle
	\begin{abstract}
		In this paper, we prove a uniform version of quantum unique ergodicity for high-frequency eigensections of Pauli-Schrödinger spin operators on a certain series of unitary flat bundles over arithmetic surfaces.
	\end{abstract}
	
	\section{Introduction}

\subsection{Background}\label{A1''}

	Shnirelman's \emph{quantum ergodicity} (QE) \cite{MR0402834} stated that on a compact Riemannian manifold whose geodesic flow is \emph{ergodic} with respect to the \emph{Liouville measure}, the eigenfunctions of the Laplacian tend to be \emph{equidistributed} except for a \emph{density zero} subsequence.

	The \emph{quantum unique ergodicity} conjecture (QUE) of Rudnick-Sarnak \cite{MR1266075} asserted that for Riemannian manifolds of negative sectional curvature, all eigenfunctions equidistribute. Lindenstrauss \cite{MR2195133} established arithmetic quantum unique ergodicity (AQUE). Arithmetic surfaces carry additional symmetries called \emph{Hecke operators}, and these operators commute with each other and with the Laplacian, so it is natural to consider the \emph{joint} eigenfunctions of the Laplacian and all Hecke operators, called \emph{Hecke-Maass forms}. The non-compact AQUE follows from combining \cite{MR2195133} with Soundararajan \cite{MR2680500}, which ruled out the escape of mass.

	In \cite{MR2838248, MR3615411}, Bismut-X. Ma-Zhang obtained the asymptotics of analytic torsions for a series of highly \emph{non-unitary} flat vector bundles, and this was further extended by the author \cite{MR4665497} for full asymptotic analytic torsions and by Puchol \cite{MR4611826} for asymptotic holomorphic torsions.
	
	Inspired by the framework and technique of \cite{MR2838248,MR3615411,MR4665497,MR4611826}, in joint work with M. Ma \cite{Ma-Ma}, we introduced a \emph{uniform version} of QE for a series of \emph{unitary} flat vector bundles. In this paper, we prove the corresponding uniform QUE for unitary flat vector bundles on arithmetic surfaces. We now explain this in detail.

	\subsection{Main results}

	\subsubsection{Geometric and analytic setup}

		Let $\mathbb{H}^2$ be the two-dimensional hyperbolic space. Let $L$ be a \emph{totally real number field} with $L\neq \mathbb{Q}$ and $D_{a,b}(L)$ a quaternion algebra over $L$ that \emph{is split at exactly one infinite place}. Let $\Gamma$ be a congruence subgroup of the reduced norm one group of $D_{a,b}(L)$ and
	\begin{equation}\label{1.1''}
		X=\Gamma\backslash{\mathbb{H}^2}
	\end{equation}
	the associated hyperbolic surface with volume form $dv_X$.
	
	Fix an infinite place $\rho\colon L\hookrightarrow \mathbb R$ distinct from the split one, then it extends to a $\mathrm{SU}_2(\mathbb{C})$-representation
	\begin{equation}\label{1.1..}
		\rho\colon \Gamma\to \mathrm{SU}_2(\mathbb{C}).
	\end{equation}
	Note that $\mathrm{SU}_2(\mathbb{C})$ acts unitarily on the complex linear space $(\mathbb{C}^2,h^{\mathbb{C}^2})$ as well as its $p$-th symmetric power $\big(\mathrm{Sym}^p(\mathbb{C}^2),h^{\mathrm{Sym}^p(\mathbb{C}^2)}\big)$ for any $p\in\mathbb{N}$, and in fact, these are all irreducible representations of $\mathrm{SU}_2(\mathbb{C})$. Then by regarding $\rho$ in \eqref{1.1..} as the holonomy of a flat principal $\mathrm{SU}_2(\mathbb{C})$-bundle over $X$, we can construct a series of unitary flat vector bundles $\{F_p\}_{p\in\mathbb{N}}$ over $X$ by
	\begin{equation}\label{a1.}
		\begin{split}
			F_p&=\Gamma\backslash\big({\mathbb{H}^2}\times \mathrm{Sym}^p(\mathbb{C}^2)\big)\\
			&=\big\{(x,v)\in{\mathbb{H}^2}\times \mathrm{Sym}^p(\mathbb{C}^2)\big\}/\big((x,\alpha)\sim(\gamma x,\rho(\gamma)\alpha)\ \text{for\ any\ }\gamma\in\Gamma\big).
		\end{split}
	\end{equation}

	Let $C^\infty({\mathbb{H}^2},\mathrm{Sym}^p(\mathbb{C}^2))$ be the space of smooth $\mathrm{Sym}^p(\mathbb{C}^2)$-valued functions on ${\mathbb{H}^2}$, and we have a $\Gamma$-action on it such that for any $\gamma\in\Gamma$ and $u\in {C}^\infty({\mathbb{H}^2},\mathrm{Sym}^p(\mathbb{C}^2))$,
	\begin{equation}\label{1.3''}
		\big(\gamma(u)\big)(x)=\rho(\gamma)u(\gamma^{-1}x),
	\end{equation}
	then it follows from \eqref{a1.} that the space ${C}^\infty(X,F_p)$ of smooth sections of $F_p$ on $X$ is isomorphic to the $\Gamma$-invariant subspace ${C}^\infty_\Gamma\big({\mathbb{H}^2},\mathrm{Sym}^p(\mathbb{C}^2)\big)$ of ${C}^\infty({\mathbb{H}^2},\mathrm{Sym}^p(\mathbb{C}^2))$, that is,
	\begin{equation}\label{1.3'}
		{C}^\infty(X,F_p)\cong {C}^\infty_\Gamma\big({\mathbb{H}^2},\mathrm{Sym}^p(\mathbb{C}^2)\big).
	\end{equation}
	Therefore, naturally we can define a flat connection $\nabla^{F_p}$ and a Hermitian metric $h^{F_p}$ on $F_p$, using the derivative $d$ on ${\mathbb{H}^2}$ and $h^{\mathrm{Sym}^p(\mathbb{C}^2)}$. Let $\langle\cdot,\cdot\rangle_{L^2(X,F_p)}$ be the $L^2$-metric on ${C}^\infty(X,F_p)$ induced by $(h^{F_p},dv_X)$.

	Let $\Delta^{F_p}$ be the nonnegative Laplacian acting on ${C}^\infty(X,F_p)$, which is also referred to as the \emph{Schrödinger-Pauli spin $p/2$ operator} or \emph{twisted Schrödinger operator}, and by \eqref{1.3'}, for local coordinates $(x_1,x_2)$ of ${\mathbb{H}^2}$, we have
	\begin{equation}
		\Delta^{F_p}=-x_2^2\big({\pa_ {x_1}^2}+{\pa_{x_2}^2}\big).
	\end{equation}
	We list all the eigenvalues $0\leqslant\lambda_{p,1}\leqslant\lambda_{p,2}\leqslant\cdots$ of $\Delta^{F_p}$ counted with multiplicity and associated orthonormal eigensections
	\begin{equation}\label{a1}
		\Delta^{F_p}u_{p,j}=\lambda_{p,j}u_{p,j},\quad\lV u_{p,j}\rV_{L^2(X,F_p)}=1.
	\end{equation}

	There is an extra family of operators acting on ${C}^\infty(X,F_p)$ called \emph{Hecke operators}, see \cref{Cc}. They commute with each other and $\Delta^{F_p}$. Therefore, we now assume further that each $u_{p,j}$ in \eqref{a1} is also an eigensection of \emph{all} Hecke operators.

	\subsubsection{Arithmetic QUE}
	
	We now state our first main result, namely AQUE for a single flat vector bundle, see Theorem \ref{A2'} for the full statement and its proof.
	\begin{theo}\label{A2}
		For any fixed $p\in\mathbb{N}$, let $(X,F_p)$ be defined in \eqref{1.1''} and \eqref{a1.}. Then for any $A\in{C}^\infty(X,\mathrm{End}(F_p))$, we have
		\begin{equation}\label{a6}
			\lim_{j\to\infty}\langle Au_{p,j},u_{p,j}\rangle_{L^2(X,F_p)}=\frac{1}{\dim_\mathbb{C} F_p\cdot \mathrm{Vol}(X)}\int_X\tro^{F_p}[A(x)]dv_X(x).
		\end{equation}
	\end{theo}
	
	Theorem \ref{A2} asserts that, for each of $\{F_p\}_{p\in\mathbb{N}}$, high-frequency eigensections satisfy an equidistribution property. It is natural to ask about the \emph{uniformity} of \eqref{a6} with respect to $p\in\mathbb{N}$. However, the first problem is how to test sections of different bundles. A natural first step is to consider the test space ${C}^\infty(X)$. Equivalently, this corresponds to checking the equidistribution of the probability measures
	\begin{equation}\label{1.9'.}
		\big\{\Vert u_{p,j}(x)\Vert_{F_p}^2dv_{X}(x)\big\}_{p,j\in\mathbb{N}}
	\end{equation}
	on $X$. Unfortunately, the above measures overlook the rich information of eigensections along the fibre, since taking the pointwise norm $\lV\cdot\rV_{F_p}$ loses too much information. Instead, we shall find a space larger than ${C}^\infty(X)$ that can capture fibrewise information. We now discuss this in detail.

	\subsubsection{Arithmetic uniform QUE}\label{s1.2.2}

	Let us begin by giving a geometric interpretation of the eigensections $u_{p,j}$ in \eqref{a1}. 
	
	We form the $3$-sphere
	\begin{equation}
		\mathbb{S}^3=\{z=(z_0,z_1)\in\mathbb{C}^2\mid\vert z_0\vert^2+\vert z_1\vert^2=1\}
	\end{equation}
	and let $dv_{\mathbb{S}^3}$ be the Euclidean measure on $\mathbb{S}^3$. Let $\langle\cdot,\cdot\rangle_{L^2(\mathbb{S}^3)}$ be the $L^2$-metric on ${C}^\infty(\mathbb{S}^3)$ induced by $dv_{\mathbb{S}^3}$. Moreover, for any $g\in\mathrm{SU}_2(\mathbb{C}), \alpha\in {C}^\infty(\mathbb{S}^3)$, we define $g\alpha\in{C}^\infty(\mathbb{S}^3)$ by
	\begin{equation}\label{1.8'}
		(g\alpha)(z_0,z_1)=\alpha(g^{-1}\cdot(z_0,z_1)).
	\end{equation}
	
	Now we view
	\begin{equation}\label{1.10...}
		\mathrm{Sym}^p(\mathbb{C}^2)\subset {C}^\infty(\mathbb{S}^3),
	\end{equation}
	as the space of homogeneous polynomial functions of degree $p$, that is,
	\begin{equation}\label{1.12}
		\mathrm{Sym}^p(\mathbb{C}^2)=\{a_0z_0^p+a_1z_0^{p-1}z_1+\cdots+a_p{z_1}^p\mid a_0,\cdots,a_p\in\mathbb{C}\}.
	\end{equation}
	Then $\langle\cdot,\cdot\rangle_{L^2(\mathbb{S}^3)}$ restricts to the metric  $h^{\mathrm{Sym}^p(\mathbb{C}^2)}$ on $\mathrm{Sym}^p(\mathbb{C}^2)$, and \eqref{1.8'} restricts to the action of $g$ on $\mathrm{Sym}^p(\mathbb{C}^2)$.

	For any $u_{p,j}\in{C}^\infty\big(X,F_p\big)$ given in \eqref{a1}, using \eqref{1.3'}, it lifts to a $\Gamma$-invariant $\mathrm{Sym}^p(\mathbb{C}^2)$-valued function on ${\mathbb{H}^2}$, and by \eqref{1.10...} and \eqref{1.12}, we can regard $u_{p,j}\in{C}^\infty({\mathbb{H}^2}\times\mathbb{S}^3)$ by
	\begin{equation}\label{1.14''}
		u_{p,j}\colon (x,z)\in {\mathbb{H}^2}\times\mathbb{S}^3\mapsto\big(u_{p,j}(x)\big)(z),
	\end{equation}
	where $\big(u_{p,j}(x)\big)\in \mathrm{Sym}^p(\mathbb{C}^2)$ and we evaluate at $z\in \mathbb{S}^3$. By \eqref{1.8'} and \eqref{1.14''}, and since $u_{p,j}$ is invariant with respect to the $\Gamma$-action given in \eqref{1.3''}, we have
	\begin{equation}\label{1.14'''}
		\begin{split}
			u_{p,j}(\gamma x,\rho(\gamma) z)&=\big(u_{p,j}(\gamma x)\big)(\rho(\gamma)z)=\big(\rho(\gamma) u_{p,j}(x)\big)\big(\rho(\gamma)z\big)\\
			&=\big(u_{p,j}(x)\big)(z)=u_{p,j}(x,z)
		\end{split}
	\end{equation}
	where $\rho(\gamma)\in\mathrm{SU}_2(\mathbb{C})$ acts naturally on $z\in\mathbb{S}^3$. Therefore, if we define a flat $\mathbb{S}^3$-bundle $\mathscr{M}$ over $X$ by
	\begin{equation}\label{1.10}
		\begin{split}
			\mathscr{M}&=\Gamma\backslash\big({\mathbb{H}^2}\times\mathbb{S}^3\big)\\
			&=\{(x,z)\in{\mathbb{H}^2}\times\mathbb{S}^3\}/\big((x,z)\sim (\gamma x,\rho(\gamma)z)\ \text{for\ any\ }\gamma\in\Gamma\big),
		\end{split}
	\end{equation}
	by \eqref{1.14'''}, $u_{p,j}$ passes to a smooth function in ${C}^\infty(\mathscr{M})$. Let $dv_{\mathscr{M}}(x,z)$ be the measure on $\mathscr{M}$ locally given by $dv_{X}(x)dv_{\mathbb{S}^3}(z)$, or equivalently, $dv_{{\mathbb{H}^2}}dv_{\mathbb{S}^3}$ is a volume form on ${\mathbb{H}^2}\times\mathbb{S}^3$, which is invariant with respect to the diagonal $\Gamma$-action as in \eqref{1.10}, so it passes to the volume form $dv_{\mathscr{M}}(x,z)$. Let $\lv\cdot\rv_{\mathbb{C}}$ be the modulus on $\mathbb{C}$. We end up with probability measures
	\begin{equation}\label{1.14'}
		\big\{\lv u_{p,j}(x,z)\rv_\mathbb{C}^2d{v}_{\mathscr{M}}(x,z)\big\}_{p,j\in\mathbb{N}}
	\end{equation}
	on $\mathscr{M}$. The measures in \eqref{1.14'} are the desired refinement of measures in \eqref{1.9'.}, in the sense that the measures in \eqref{1.14'} project along the $\mathbb{S}^3$-fibre to the measures in \eqref{1.9'.}, see \cref{S2.2..} for more details.

	Now we can state our second main result, arithmetic uniform QUE, which asserts the equidistribution property of the measures in \eqref{1.14'}, in a \emph{uniform} manner with respect to $p\in\mathbb{N}$, and see Theorem \ref{E6''} for a full version with momentum variables and its proof.
	
	\begin{theo}\label{C9'}
		For measures constructed in \eqref{1.14'}, we have for any $\mathscr{A}\in{C}^\infty(\mathscr{M})$,
		\begin{equation}\label{9.}
			\lim_{\lambda\to\infty}\sup_{\substack{(p,j)\in\mathbb{N}^{2},\\
					\lambda_{p,j}\geqslant \lambda}}\lv\int_\mathscr{M}\mathscr{A}\lv u_{p,j}\rv_\mathbb{C}^2d{v}_{\mathscr{M}}-\frac{1}{\mathrm{Vol}(\mathscr{M})}\int_{\mathscr{M}}\mathscr{A}dv_{\mathscr{M}}\rv=0.
		\end{equation}
	\end{theo}
	
	Theorem \ref{C9'} asserts that $\lv u_{p,j}\rv_\mathbb{C}^2d{v}_{\mathscr{M}}$ tends to be equidistributed as long as its corresponding eigenvalue $\lambda_{p,j}$ is large, regardless of the spin quantum number $p\in\mathbb{N}$. We emphasize that it is this uniformity that makes Theorem \ref{C9'} noteworthy, and when we fix a $p\in\mathbb{N}$ in \eqref{9.}, it is equivalent to \eqref{a6}, see \cref{s3.2} for details.

	\subsection{Main techniques}

We emphasize that the dynamical ingredients used in this paper are not
new, they are drawn from existing works and applied here in a new
geometric setting.

The proofs of Theorems \ref{A2} and \ref{C9'} closely follow the strategy developed by Lindenstrauss \cite{MR2195133} for the AQUE conjecture and proceed in three steps, geodesic invariance, Hecke recurrence and positive entropy. The notes of Einsiedler-Ward \cite{Arizona} provide an accessible introduction to Lindenstrauss’s approach, and much of our argumentation follows this source, with adaptations needed to fit our setting.

	We note that \cite{MR2195133} treats congruence lattices over $\mathbb{Q}$, whereas our setting involves congruence lattices over a totally real number field $L\neq\mathbb{Q}$. Accordingly, we rely on the extension of Lindenstrauss’s work by Shem-Tov-Silberman \cite{MR5036694} for general number fields. Finally, our use of the positive entropy argument is largely based on the approaches of Bourgain-Lindenstrauss \cite{MR1957735} and Silberman-Venkatesh \cite{MR4033919}.

	We now explain how the study of endomorphisms of vector bundles can be reduced to that of functions.

	\subsubsection{Arithmetic QUE}\label{A.3.1}

	Consider the unit tangent bundle $U(X)$ of $X$ given by
	\begin{equation}\label{1.19}
		U(X)\cong\Gamma\backslash \mathrm{SL}_2(\mathbb{R}),
	\end{equation}
	together with a natural projection map $\pi\colon U(X)\to X$ and the \emph{Liouville volume form} $dv_{U(X)}$, passing from the $\Gamma$-quotient of the Haar volume form $dv_{\mathrm{SL}_2(\mathbb{R})}$ on $\mathrm{SL}_2(\mathbb{R})$. Let us \emph{fix} a $p\in\mathbb{N}$, and put
	\begin{equation}\label{1.20}
		\pi^*F_p=\Gamma\backslash(\mathrm{SL}_2(\mathbb{R})\times \mathrm{Sym}^p(\mathbb{C}^2)),
	\end{equation}
	the pull-back flat vector bundle of $F_p$ over $U(X)$. 
	
	The \emph{geodesic flow} $(g_t)$ on $U(X)$ is induced by the right action of 
	\begin{equation}\label{2.12}
		g_t=\left(\begin{matrix}e^{t/2}&0\\0&e^{-t/2}\end{matrix}\right)
	\end{equation}
	on $\mathrm{SL}_2(\mathbb{R})$. Using semiclassical analysis, any weak star limit $L_{p,X}$ of functionals
	\begin{equation}\label{1.22}
		\big\{A\in{C}^\infty\big(X,\mathrm{End}(F_p)\big)\mapsto\langle A u_{p,j},u_{p,j}\rangle_{L^2(X,F_p)}\in\mathbb{C}\big\}_{j\in\mathbb{N}}
	\end{equation}
	when $\lambda_{p,j}\to\infty$ can be lifted to a \emph{geodesic invariant} functional
	\begin{equation}
		L_{U(X),p}\colon C^\infty\big(U(X),\pi^*\mathrm{End}({F}_p)\big)\to \mathbb{C}.
	\end{equation}
	We shall derive \eqref{a6} by proving for any $A\in{C}^\infty(U(X),\pi^*\mathrm{End}({F}_p))$,
	\begin{equation}\label{1.18''}
		L_{U(X),p}(A)=\frac{1}{\dim_\mathbb{C} F_p\cdot \mathrm{Vol}(U(X))}\int_{U(X)}\tro^{\pi^*F_p}[A]dv_{U(X)}.
	\end{equation}

	To get \eqref{1.18''}, a key simplification is that by the Borel density theorem \cite[Corollary 4.5.6]{MR3307755}, the image $\rho(\Gamma)\subset \mathrm{SU}_2(\mathbb{C})$ of the map given in \eqref{1.1..} is \emph{dense}. Therefore, it is sufficient to show that the restriction of $L_{U(X),p}$ to ${C}^\infty(U(X))\cdot\mathrm{Id}_{\pi^*F_p}$ is the normalized Liouville measure $\frac{1}{\mathrm{Vol}(U(X))}dv_{U(X)}$, see the proof of Theorem \ref{E5'} for more details. In other words, we shall prove that for any $f\in {C}^\infty(U(X))$,
	\begin{equation}\label{1.18..}
		L_{U(X),p}\big(f\mathrm{Id}_{F_p}\big)=\frac{1}{\mathrm{Vol}(U(X))}\int_{U(X)}fdv_{U(X)}.
	\end{equation}

	\subsubsection{Arithmetic Uniform QUE}\label{S1.3.2}

	Recall the flat $\mathbb{S}^3$-bundle $\mathscr{M}$ over $X$ given in \eqref{1.10}. We then form its pullback $\mathbb{S}^3$-bundle $\pi^*\mathscr{M}$ over $U(X)$ by
	\begin{equation}\label{1.25}
		\pi^*\mathscr{M}\cong \Gamma\backslash(\mathrm{SL}_2(\mathbb{R})\times\mathbb{S}^3)
	\end{equation}
	with a volume form $dv_{\pi^*\mathscr{M}}$ locally given by $dv_{U(X)}dv_{\mathbb{S}^3}$, passing from the $\Gamma$-quotient of $dv_{\mathrm{SL}_2(\mathbb{R})}dv_{\mathbb{S}^3}$.

	For any weak star limit $\mu_{\mathscr{M}}$ of measures $\{\lv u_{p,j}\rv_\mathbb{C}^2d{v}_{\mathscr{M}}\}_{p,j\in\mathbb{N}}$ on $\mathscr{M}$ given in \eqref{1.14'} when $\lambda_{p,j}\to\infty$, we can construct its microlocal lift measure $\mu_{\pi^*\mathscr{M}}$ on $\pi^*\mathscr{M}$, which is invariant with respect to the \emph{horizontal geodesic flow} on $\pi^*\mathscr{M}$, denoted by $(\bm{g}_t)$, defined by the right action of $g_t$ given in \eqref{2.12} on the $\mathrm{SL}_2(\mathbb{R})$ component of $\pi^*\mathscr{M}$, that is,
	\begin{equation}\label{1.26}
		\bm{g}_t\colon (y,z)\in\pi^*\mathscr{M} \rightarrow (yg_t,z)\in \pi^*\mathscr{M}.
	\end{equation}
	
	Compared with the semiclassical analysis on a single vector bundle in \cref{A.3.1}, the main difference here is that $p$ is \emph{not} fixed. We must therefore keep track of uniformity over the infinite family of bundles. This is realized by carrying out analysis on \emph{an infinite dimensional vector bundle} using \eqref{1.10...}. We follow the representation-theoretic construction of Wolpert \cite{MR1838659} and Lindenstrauss \cite{MR1859345}, which is an alternative to the original pseudo-differential approach of Schnirelman \cite{MR0402834}, Colin de Verdière \cite{MR818831} and Zelditch \cite{MR916129}.

	We shall show that
	\begin{equation}\label{1.23''}
		\mu_{\pi^*\mathscr{M}}=\frac{1}{\mathrm{Vol}(\pi^*\mathscr{M})}d{v}_{\pi^*\mathscr{M}}.
	\end{equation}
	Similar to the reduction of \eqref{1.18''} to \eqref{1.18..}, to establish \eqref{1.23''}, it is suﬀicient to check that the restriction of $\mu_{\pi^*\mathscr{M}}$ to ${C}^\infty(U(X))$ is the normalized Liouville measure $\frac{1}{U(X)}dv_{U(X)}$, see also the proof of Theorem \ref{E5'}. In other words, for any $f\in{C}^\infty(U(X))$,
	\begin{equation}\label{1.24''}
		\mu_{\pi^*\mathscr{M}}(f)=\frac{1}{\mathrm{Vol}(U(X))}\int_{U(X)}fd{v}_{U(X)}.
	\end{equation}

	\subsection{Mixed quantization}

	In \cite[\S\,4]{Ma-Ma}, the microlocal lift $\mu_{\pi^*\mathscr{M}}$ and the Uniform QE were established using the general \emph{mixed quantization} framework. Subsequently, Cekić-Lefeuvre \cite[Theorem 5.1.7]{cekić2024semiclassicalanalysisprincipalbundles} obtained a related but weaker, non-uniform version of QE.
	
Mixed quantization is the composition of the following maps
	\begin{equation}\label{diag1}
		{C}^\infty(\pi^*\mathscr{M})\xrightarrow{T_{\cdot,p}}{C}^\infty\big(U(X),\pi^*\mathrm{End}(F_p)\big)\xrightarrow{\mathrm{Op}_h(\cdot)}\mathrm{End}\big(L^2(X,F_p)\big),
	\end{equation}
	where $T_{\cdot,p}$ is the \emph{Berezin-Toeplitz quantization} along the fibre direction and controls  the behavior of an infinite number of linear spaces, while $\operatorname{Op}_h$ is the \emph{Weyl quantization} along
	the base manifold $X$ and controls the high frequency eigensections. Strictly speaking, this is a slight abuse of terminology, since the Berezin-Toeplitz quantization is performed along $\mathbb{S}^3/\mathbb{S}^1\cong\mathbb{CP}^1$ rather than $\mathbb{S}^3$. Combining these quantizations enables simultaneous control of the high-frequency eigensections of an infinite number of bundles. For more details on Berezin-Toeplitz and Weyl quantizations, we refer to X. Ma-Marinescu \cite[\S\,7]{MR2339952} and Zworski \cite[\S\,4]{MR2952218} respectively. The measures in \eqref{1.14'} correspond to $T_{\cdot,p}$, and the microlocal lift $\mu_{\pi^*\mathscr{M}}$ corresponds to $\mathrm{Op}_h(\cdot)$. More precisely, $\mu_{\pi^*\mathscr{M}}$ is a weak-star limit of the following functionals
	\begin{equation}
		\Big\{\mathscr{B}\in{C}^\infty(\pi^*\mathscr{M})\mapsto \big\langle \mathrm{Op}_{\lambda_{p,j}^{-1/2}}(T_{\mathscr{B},p}) u_{p,j},u_{p,j}\big\rangle_{L^2(X,F_p)}\Big\}_{p,j},
	\end{equation}
	when $\lambda_{p,j}\to\infty$.
	
Since the geometric setting of this paper is quite specific, we do not attempt to introduce the mixed quantization framework in full generality. Rather, we present a very special case sufficient for our purposes, namely to obtain uniform estimates and to show that Theorem \ref{C9'} implies Theorem \ref{A2}, see \cref{s2.4}, \cref{s2.5}, \cref{s2.6}, \cref{s3.2} and \cref{s3.3'} for details.

	\subsection{Organization of the paper}

	This paper is organized as follows. In \Cref{A}, we reinterpret the geometric setup using an infinite dimensional flat vector bundle $\mathscr{F}$ over $X$. In \cref{S3AU}, we prove the main results, Theorems \ref{A2} and \ref{C9'}, under several properties that will be verified in later sections. In \Cref{B}, we construct microlocal lifts of high frequency eigensections of the Laplacian $\Delta^{\mathscr{F}}$ acting on ${C}^\infty(X,\mathscr{F})$ and establish their geodesic invariant property. In \cref{C}, we introduce and study Hecke operators acting on $\mathscr{F}$. In \cref{D}, we prove the Hecke recurrence property for microlocal lifts. In \cref{S7}, we establish the positive entropy property for microlocal lifts.

	\subsection{Notation}\label{s1.6}

	Recall the geometric setup of \eqref{1.1''}, \eqref{a1.}, \eqref{1.10}, \eqref{1.19}, \eqref{1.20} and \eqref{1.25}.	Then we collect the notation used in this section here and keep it for the remainder of the paper.
	
	We use $x$ to denote a point of $\mathbb{H}^2$ or $X$, $y$ a point of $\mathrm{SL}_2(\mathbb{R})$ or $U(X)$, and $z$ a point of $\mathbb{S}^3$. Accordingly, a point of $\mathscr{M}$ will be written as $\Gamma(x,z)$, or simply $(x,z)$ when no ambiguity arises. Similarly, a point of $\pi^*\mathscr{M}$ will be written as $\Gamma(y,z)$, or $(y,z)$. 
	
	Moreover, we use $u, s, \alpha, A, B, \mathscr{A}$ and $\mathscr{B}$, possibly with subscripts, to denote sections or functions of the following types, $u\in C^\infty(X,{F}_p)$ or $C^\infty(X,\mathscr{F})$, $s\in C^\infty(U(X),\pi^*{F}_p)$ or $C^\infty(U(X),\pi^*\mathscr{F})$, $\alpha\in \mathrm{Sym}^p(\mathbb{C}^2)$ or $C^\infty(\mathbb{S}^3)$, $A\in C^\infty(X,\mathrm{End}({F}_p))$, $B\in C^\infty(U(X),\pi^*\mathrm{End}({F}_p))$, $\mathscr{A}\in{C}^\infty(\pi^*\mathscr{M})$ and $\mathscr{B}\in{C}^\infty(\pi^*\mathscr{M})$, where $\mathscr{F}$ is the infinite dimensional flat vector bundle to be defined in \cref{A}.
	
	Throughout the paper, the symbol $C$ denotes a positive constant. When the dependence of $C$ on certain parameters is important, this will be indicated by a subscript. The value of $C$ may change from line to line, even when the same notation is used.

	\subsection{Acknowledgment}

	I am grateful to Nigel Higson and Guoliang Yu for support and understanding, and to Svetlana Katok for encouragement. I would like to express gratitude to Federico Rodriguez Hertz for pointing out the essential simplification steps \eqref{1.18..} and \eqref{1.24''}, which reduce the test spaces ${C}^\infty(U(X),\pi^*\mathrm{End}(F_p))$ and ${C}^\infty(\pi^*\mathscr{M})$ to the much simpler space ${C}^\infty(U(X))$, making it possible to use the machinery of Lindenstrauss. I would like to thank Snir Ben Ovadia and Louis Ioos for helpful discussions,
	Elon Lindenstrauss for pointing me to Shem-Tov-Silberman \cite{MR5036694}, and John Voight for valuable explanations on quaternionic arithmetic surfaces. I thank the referees for careful reading, useful suggestions and corrections. I was supported by the NSF grants DMS-1952669 and DMS-1546917.

	\section{Infinite Dimensional Flat Vector bundles}\label{A}

	In this section, motivated by \eqref{1.12}, we introduce infinite dimensional counterparts of the finite dimensional geometric objects considered in this paper, such as $\{F_p\}_{p\in\mathbb{N}}$ and $\{\mathrm{End}(F_p)\}_{p\in\mathbb{N}}$. In \cref{s2.1}, we introduce an infinite dimensional vector bundle $\mathscr{F}$. In \cref{S2.2..}, we equip $\mathscr{F}$ with a connection, a metric, and a Laplacian operator. In \cref{s2.3}, we embed $\{F_p\}_{p\in\mathbb{N}}$ into $\mathscr{F}$. In \cref{s2.4}, we recall the Berezin-Toeplitz quantization. In \cref{s2.5}, we introduce a surjective projection from $\mathscr{F}$ to $\{\mathrm{End}(F_p)\}_{p\in\mathbb{N}}$ via quantization. In \cref{s2.6}, we build a correspondence between measures in \eqref{1.14'} and functionals in \eqref{1.22}. In \cref{s2.7}, we discuss the role of uniformity in our estimates.

	We follow the notational conventions introduced in \cref{s1.6}.

	\subsection{An infinite dimensional flat vector bundle}\label{s2.1}

	Recall the definition of $\mathscr{M}$ in \eqref{1.10}, then for $u\in {C}^\infty(\mathscr{M})$, we regard it as a $\Gamma$-invariant function $u\in {C}^\infty({\mathbb{H}^2}\times \mathbb{S}^3)$, that is,
	\begin{equation}\label{2.1..}
		u(\gamma x,\rho(\gamma) z)=u(x,z).
	\end{equation}
	Now we view $u\in{C}^\infty({\mathbb{H}^2},{C}^\infty(\mathbb{S}^3))$ by $u(x)=u(x,\cdot)\in {C}^\infty(\mathbb{S}^3)$, then from \eqref{2.1..}, $u$ is $\Gamma$-invariant with respect to the action
	\begin{equation}
		\big(\gamma(u)\big)(x)=\rho(\gamma)u(\gamma^{-1}x),
	\end{equation}
	where $\rho(\gamma)$ acts on $u(\gamma^{-1}x)$ through \eqref{1.8'}. To verify this, we compute that
	\begin{equation}
		\big(\rho(\gamma)u(x)\big)(z)=u(x)\big(\rho(\gamma)^{-1}z\big)=u\big(x,\rho(\gamma)^{-1}z\big)=u(\gamma x,z)=u(\gamma x)(z).
	\end{equation}
	Therefore, let ${C}^\infty_\Gamma\big({\mathbb{H}^2},{C}^\infty(\mathbb{S}^3)\big)\subset {C}^\infty\big({\mathbb{H}^2},{C}^\infty(\mathbb{S}^3)\big)$ denote the $\Gamma$-invariant subspace, then analogous to  \eqref{1.3'}, we have the following isomorphism 
	\begin{equation}\label{2.5}
		{C}^\infty(\mathscr{M})\cong{C}^\infty_\Gamma\big({\mathbb{H}^2},{C}^\infty(\mathbb{S}^3)\big).
	\end{equation}

	We form an infinite dimensional flat vector bundle $\mathscr{F}$ over $X$ with fibre $C^\infty(\mathbb{S}^3)$ by
	\begin{equation}\label{2.4..}
		\begin{split}
			\mathscr{F}&=\Gamma\backslash\big({\mathbb{H}^2}\times{C}^\infty(\mathbb{S}^3)\big)\\
			&=\big\{(x,\alpha)\in {\mathbb{H}^2}\times{C}^\infty(\mathbb{S}^3)\big\}/\big((x,\alpha)\sim (\gamma x,\rho(\gamma)\alpha) \ \text{for\ any\ }\gamma\in\Gamma\big),
		\end{split}
	\end{equation}
	then we get from \eqref{2.5}
	\begin{equation}\label{2.5..}
		{C}^\infty (\mathscr{M})\cong{C}^\infty(X,\mathscr{F}).
	\end{equation}
	
	Recall $\pi^*\mathscr{M}$ defined in \eqref{1.25}. Let $\pi^*\mathscr{F}$ be the pull back bundle of $\mathscr{F}$ over $U(X)$. Similar to \eqref{2.5..}, we have
	\begin{equation}\label{2.13'}
		{C}^\infty(\pi^*\mathscr{M})\cong C^\infty(U(X),\pi^*\mathscr{F}).
	\end{equation}
	
	For convenience, we summarize the relevant geometric objects in the following diagram.
	\begin{equation}\label{diag1'}
		\begin{tikzcd}[ampersand replacement=\&, column sep=normal,row sep=normal]
			\pi^*\mathscr{F}=\Gamma\backslash \big(\mathrm{SL}_2(\mathbb{R})\times{C}^\infty(\mathbb{S}^3)\big)\arrow[dr,""swap] \&	\pi^*\mathscr{M}=\Gamma\backslash(\mathrm{SL}_2(\mathbb{R})\times\mathbb{S}^3)\arrow[d,"q"]\arrow[dd, bend left=75, "\pi"] \\
			\&	U(X)=\Gamma\backslash \mathrm{SL}_2(\mathbb{R})\arrow[dd, bend left=75, "\pi"] \\
			\mathscr{F}=\Gamma\backslash \big({\mathbb{H}^2}\times{C}^\infty(\mathbb{S}^3)\big)	\arrow[rd,""]  \&	\mathscr{M}=\Gamma\backslash({\mathbb{H}^2}\times\mathbb{S}^3)\arrow[d,"q"]  \\
			\&	X=\Gamma\backslash{\mathbb{H}^2}   
		\end{tikzcd}.
	\end{equation}
	
		Recall the notation introduced in \cref{s1.6}. Note that although the isomorphisms in \eqref{2.5..} and \eqref{2.13'} hold, the pairs $(u,s)$ and $(\mathscr{A},\mathscr{B})$ play different roles in practice, $(u,s)$ are sections of the flat vector bundle, whereas $(\mathscr{A},\mathscr{B})$ are typically used as test functions to detect equidistribution.

	\subsection{Connection, metric, and Laplacian}\label{S2.2..}
	
	Let $h^{\mathscr{F}}$ be the metric on $\mathscr{F}$ induced by $\langle\cdot,\cdot\rangle_{L^2(\mathbb{S}^3)}$ and $\nabla^{\mathscr{F}}$ the natural flat connection on $\mathscr{F}$, then $h^{\mathscr{F}}$ is parallel with respect to $\nabla^\mathscr{F}$. 
	
	Let $\langle\cdot,\cdot\rangle_{L^2(X,\mathscr{F})}$ be the $L^2$-metric on ${C}^\infty(X,\mathscr{F})$ induced by $(h^{\mathscr{F}},dv_{X})$, and let $\langle\cdot,\cdot\rangle_{L^2(\mathscr{M})}$ be the $L^2$-metric on ${C}^\infty(\mathscr{M})$ induced by $dv_{\mathscr{M}}$, then we have
	\begin{equation}\label{2.9..}
		\langle\cdot,\cdot\rangle_{L^2(X,\mathscr{F})}=\langle\cdot,\cdot\rangle_{L^2(\mathscr{M})}.
	\end{equation} 
	To see this, for any $u\in {C}^\infty(X,\mathscr{F})$, using the two ways in \eqref{2.5..} it is viewed, we can take pointwise norms with respect to $\mathscr{F}$ and $\mathbb{C}$, obtaining smooth functions
	\begin{equation}\label{2.10}
		\lv u\rv^2_{\mathscr{F}}\in{C}^\infty(X),\quad\lv u\rv^2_{\mathbb{C}}\in{C}^\infty(\mathscr{M})
	\end{equation}
	accordingly. We compute that locally
	\begin{equation}\label{2.10'}
		\begin{split}
			\lv u\rv_{L^2(X,\mathscr{F})}^2&=\int_X\lv u(x)\rv_{\mathscr{F}}^2dv_X(x)=\int_X\Big(\int_{\mathbb{S}^3}\lv u(x,z)\rv^2 dv_{\mathbb{S}^3}(z)\Big)dv_X(x)\\
			&=\int_{\mathscr{M}}\lv u(x,z)\rv^2dv_\mathscr{M}(x,z)=\lv u\rv_{L^2(\mathscr{M})}^2.
		\end{split}
	\end{equation}

	Similarly, from \eqref{2.13'}, for any $s\in {C}^\infty(U(X),\pi^*\mathscr{F})$, by taking pointwise norms, we get
	\begin{equation}\label{2.12''}
		\lv s\rv^2_{\pi^*\mathscr{F}}\in{C}^\infty(U(X)),\quad\lv s\rv^2_{\mathbb{C}}\in{C}^\infty(\pi^*\mathscr{M})
	\end{equation}
	accordingly, then calculate along \eqref{2.10'}, we get
	\begin{equation}\label{2.13''}
		\langle\cdot,\cdot\rangle_{L^2(U(X),\pi^*\mathscr{F})}=\langle \cdot,\cdot\rangle_{L^2(\pi^*\mathscr{M})}.
	\end{equation}

	Let $\Delta^{\mathscr{F}}$ be the Laplacian acting on ${C}^\infty(X,\mathscr{F})$, then for any $u\in{C}^\infty(X,\mathscr{F})$, by viewing $u\in{C}^\infty_\Gamma\big({\mathbb{H}^2},{C}^\infty(\mathbb{S}^3)\big)$ through \eqref{2.5}, we have
	\begin{equation}\label{2.14}
		\Delta^{\mathscr{F}}u=-x_2^2\big({\pa_ {x_1}^2}+{\pa_{x_2}^2}\big)u.
	\end{equation}
	Note that by \eqref{2.5..}, the operator $\Delta^{\mathscr{F}}$ can be regarded as the partial Laplacian acting on ${C}^\infty(\mathscr{M})$ along the $\mathbb{H}^2$ component.

	\subsection{Embedding finite dimensional vector bundles}\label{s2.3}

	By \eqref{1.10...} and \eqref{1.12}, we can view
	\begin{equation}\label{2.15}
		\Big(\mathrm{Sym}^p(\mathbb{C}^2),h^{\mathrm{Sym}^p(\mathbb{C}^2)}\Big)\longhookrightarrow \Big(C^\infty(\mathbb{S}^3),h^{L^2(\mathbb{S}^3)}\Big)
	\end{equation}
	as the subspace of homogeneous polynomial functions of degree $p$. From \eqref{a1.}, \eqref{1.10...}, \eqref{2.4..}, we have an embedding
	\begin{equation}\label{2.16'}
		\begin{split}
			&\Big(\nabla^{F_p},h^{F_p},\Delta^{F_p},{C}^\infty(X,F_p),{C}^\infty\big(U(X),\pi^*F_p\big)\Big)\\
			&\longhookrightarrow\Big(\nabla^{\mathscr{F}},h^{\mathscr{F}},\Delta^{\mathscr{F}},{C}^\infty(X,\mathscr{F}),{C}^\infty\big(U(X),\pi^*\mathscr{F}\big)\Big).
		\end{split}
	\end{equation}

	Also, we can extend the diagram \eqref{diag1'} to the following
	\begin{equation}\label{diag2}
		\begin{tikzcd}[ampersand replacement=\&, column sep=normal,row sep=normal]
			\pi^*\mathscr{F}=\Gamma\backslash \big(\mathrm{SL}_2(\mathbb{R})\times{C}^\infty(\mathbb{S}^3)\big)\arrow[dr,""swap] \&	\pi^*\mathscr{M}=\Gamma\backslash(\mathrm{SL}_2(\mathbb{R})\times\mathbb{S}^3)\arrow[d,"q"]\arrow[dd, bend left=75, "\pi"] \\
			\pi^*F_p=\Gamma\backslash \big(\mathrm{SL}_2(\mathbb{R})\times\mathrm{Sym}^p(\mathbb{C}^2)\big) \arrow[r,""swap]\arrow[u,hook,""]\&	U(X)=\Gamma\backslash \mathrm{SL}_2(\mathbb{R})\arrow[dd, bend left=75, "\pi"] \\
			\mathscr{F}=\Gamma\backslash \big({\mathbb{H}^2}\times{C}^\infty(\mathbb{S}^3)\big)	\arrow[rd,""]  \&	\mathscr{M}=\Gamma\backslash({\mathbb{H}^2}\times\mathbb{S}^3)\arrow[d,"q"]  \\
			F_p=\Gamma\backslash\big({\mathbb{H}^2}\times\mathrm{Sym}^p(\mathbb{C}^2)\big)\arrow[u,hook,""] \arrow[r,""swap]\&	X=\Gamma\backslash{\mathbb{H}^2}   
		\end{tikzcd}.
	\end{equation}
	
	Using \eqref{2.16'}, we embed all finite dimensional flat vector bundles $\{F_p\}_{p\in\mathbb{N}}$, together with their metrics and Laplace operators, into a single infinite dimensional flat vector bundle $\mathscr{F}$. The only remaining object to handle is the family of endomorphism bundles $\{\mathrm{End}(F_p)\}_{p\in\mathbb{N}}$, which shall be treated in the next subsection using geometric quantization.

	\subsection{Berezin-Toeplitz quantization}\label{s2.4}

	The \emph{Berezin-Toeplitz quantization} procedure
	\begin{equation}\label{2.18}
		T_{\cdot,p}\colon f\in C^\infty(\mathbb{S}^3)\mapsto T_{f,p}\in \mathrm{End}\big(\mathrm{Sym}^p(\mathbb{C}^2)\big)
	\end{equation}
	associates to each function $f\in C^\infty(\mathbb{S}^3)$ a series of operators $\big\{T_{f,p}\in \mathrm{End}\big(\mathrm{Sym}^p(\mathbb{C}^2)\big)\big\}_{p\in\mathbb{N}}$. Given $\alpha\in\mathrm{Sym}^p(\mathbb{C}^2)$, we define $T_{f,p}\alpha$ as follows. We first multiply $\alpha$ by $f$, obtaining the function $f\alpha$, which in general does not lie in $\mathrm{Sym}^p(\mathbb{C}^2)$. We then project $f\alpha$ orthogonally onto $\mathrm{Sym}^p(\mathbb{C}^2)$ to obtain $T_{f,p}\alpha$. Equivalently, we can describe $T_{f,p}$ by
	\begin{equation}\label{2.19'}
		\begin{split}
			&\big\langle T_{f,p}\alpha,\alpha'\big\rangle_{\mathrm{Sym}^p(\mathbb{C}^2)}=\int_{\mathbb{S}^3}\big(T_{f,p}\alpha\big)(z)\overline{\alpha(z)}dv_{\mathbb{S}^3}(z)\\
			&=\big\langle f\alpha,\alpha'\big\rangle_{L^2(\mathbb{S}^3)}=\int_{\mathbb{S}^3}f(z)\alpha(z)\overline{\alpha(z)}dv_{\mathbb{S}^3}(z)
		\end{split}
	\end{equation}
	for any $\alpha,\alpha'\in \mathrm{Sym}^p(\mathbb{C}^2)$, where $dv_{\mathbb{S}^3}$ is the Euclidean measure on $\mathbb{S}^3$.

	\begin{prop}
		For any $f\in C^\infty(\mathbb{S}^3)$ and $p\in\mathbb{N}$, we have a trace formula	
		\begin{equation}\label{2.20}
			\begin{split}
				\frac{1}{\dim_{\mathbb{C}}\mathrm{Sym}^p(\mathbb{C}^2)}\mathrm{Tr}^{\mathrm{Sym}^p(\mathbb{C}^2)}[T_{f,p}]=\frac{1}{\mathrm{Vol}(\mathbb{S}^3)}\int_{\mathbb{S}^3}f(z)dv_{\mathbb{S}^3}(z).
			\end{split}
		\end{equation}
	\end{prop}

	\begin{pro}
		We have the following classical orthonormal basis of $\mathrm{Sym}^p(\mathbb{C}^2)\subset L^2(\mathbb{S}^3)$,
		\begin{equation}\label{2.19}
			\bigg\{\bigg(\frac{(p+1)\tbinom{p}{j}}{\mathrm{Vol}(\mathbb{S}^3)}\bigg)^{1/2}z_0^{p-j}z_1^{j}\bigg\}_{0\leqslant j\leqslant p}.
		\end{equation}
		Let us denote the $j$-th element of \eqref{2.19} by $\alpha_j(z)$, then by \eqref{2.19'} we compute that
		\begin{equation}\label{2.22}
			\begin{split}
				\mathrm{Tr}^{\mathrm{Sym}^p(\mathbb{C}^2)}[T_{f,p}]&=\sum_{j=0}^{p}\big\langle T_{f,p}\alpha_j,\alpha_j\Big\rangle_{\mathrm{Sym}^p(\mathbb{C}^2)}\\
				&=\sum_{j=0}^{p}\int_{\mathbb{S}^3}f(z)\alpha_j(z)\overline{\alpha_j(z)}dv_{\mathbb{S}^3}(z)=\int_{\mathbb{S}^3}f(z)\Big(\sum_{j=0}^{p}\alpha_j(z)\overline{\alpha_j(z)}\Big)dv_{\mathbb{S}^3}(z).
			\end{split}
		\end{equation}
		Furthermore,
		\begin{equation}\label{2.23}
			\begin{split}
				\sum_{j=0}^{p}\alpha_j(z)\overline{\alpha_j(z)}&=\sum_{j=0}^{p}\frac{(p+1)}{\mathrm{Vol}(\mathbb{S}^3)}\tbinom{p}{j}z_0^{p-j}z_1^{j}\overline{z_0}^{p-j}\overline{z_1}^{j}\\
				&=\frac{(p+1)}{\mathrm{Vol}(\mathbb{S}^3)}\big(\lv z_0\rv^2+\lv z_1\rv^2\big)^p=\frac{(p+1)}{\mathrm{Vol}(\mathbb{S}^3)}.
			\end{split}
		\end{equation}
		Combining \eqref{2.22} and \eqref{2.23}, we obtain \eqref{2.20}.
	\end{pro}

	\begin{prop}\label{p2.2}
		For any $p\in\mathbb{N}$, the quantization \eqref{2.18} is surjective.
	\end{prop}
	
	\begin{pro}
		We claim that the functions $\{\alpha_j(z)\overline{\alpha_k(z)}\}_{j,k=0}^p$ are linearly independent. Assuming this, by choosing $f$ in the space spanned by $\{\alpha_j(z)\overline{\alpha_k(z)}\}_{j,k=0}^p$ and solving a linear equation system, we can realize any matrix in $\mathrm{End}\big(\mathrm{Sym}^p(\mathbb{C}^2)\big)$ as the quantization of a smooth function.
		
		We now prove the linear independence. Suppose that for a family of complex numbers $\{a_{j,k}\}_{j,k=0}^p$, we have
		\begin{equation}
			\sum_{j,k}a_{j,k}z_0^{p-j}z_1^{j}\overline{z_0}^{p-k}\overline{z_1}^{k}=0
		\end{equation}
		for all $z=(z_0,z_1)\in \mathbb{S}^3$. By homogeneity, the same identity holds for all $z=(z_0,z_1)\in \mathbb{C}^2$. Taking $z_0=1$ and $z_1=re^{i\theta}$, we get
		\begin{equation}
			\sum_{j,k}a_{j,k}r^{j+k}e^{i\theta(j-k)}=0
		\end{equation}
		for all $r\geqslant 0$ and $\theta\in [0,2\pi]$. Group terms according to the Fourier frequency $j-k=\ell$, and using the linear independence of $\{e^{i\ell\theta}\}_\ell$, we must have
		\begin{equation}
			\sum_{j-k=\ell}a_{j,k}r^{j+k}=	\sum_{k}a_{k+\ell,k}r^{2k+\ell}=0
		\end{equation}
		for all $r\geqslant0$. Since $\{r^{2k+\ell}\}_{k}$ are linearly independent, it follows that $a_{j,k}\equiv0$.\qed
	\end{pro}

	\subsection{Fibrewise quantization}\label{s2.5}

	Instead of an embedding as in \eqref{2.16'}, combining \eqref{2.5..}, \eqref{2.13'}, and \eqref{diag1'}, and applying fibrewise quantization \eqref{2.18}, we get a projection
	\begin{equation}\label{2.27}
		\begin{split}
			\Big(C^\infty(\mathscr{M}),C^\infty(\pi^*\mathscr{M})\Big)&\cong\Big(C^\infty(X,\mathscr{F}),C^\infty(U(X),\pi^*\mathscr{F})\Big)\\
			&\xrightarrow[]{T_{\cdot,p}} \Big(C^\infty\big(X,\mathrm{End}(F_p)\big),C^\infty\big(U(X),\pi^*\mathrm{End}(F_p)\big)\Big),
		\end{split}
	\end{equation} 
	which is surjective by Proposition \ref{p2.2}. By \eqref{2.19'}, equivalently, for any $\mathscr{B}\in C^\infty(\pi^*\mathscr{M})$ and $s,s'\in C^\infty(U(X),\pi^*F_p)$, we have
	\begin{equation}\label{2.28}
		\big\langle T_{\mathscr{B},p}s,s'\big\rangle_{L^2(U(X),\pi^*F_p)}=\big\langle \mathscr{B}s,s'\big\rangle_{L^2(U(X),\pi^*\mathscr{F})}=\big\langle \mathscr{B}s,s'\big\rangle_{L^2(\pi^*\mathscr{M})}.
	\end{equation}

	Then, similar to \eqref{diag2}, we can extend \eqref{diag1'} to the following diagram
	\begin{equation}\label{diag}
		\begin{tikzcd}[ampersand replacement=\&, column sep=normal,row sep=normal]
			\pi^*\mathscr{F}=\Gamma\backslash \big(\mathrm{SL}_2(\mathbb{R})\times{C}^\infty(\mathbb{S}^3)\big)\arrow[dr,""swap]\arrow[d,"T_{\cdot,p}"swap] \&	\pi^*\mathscr{M}=\Gamma\backslash(\mathrm{SL}_2(\mathbb{R})\times\mathbb{S}^3)\arrow[d,"q"]\arrow[dd, bend left=75, "\pi"] \\
			\pi^*\mathrm{End}(F_p)=\Gamma\backslash \big(\mathrm{SL}_2(\mathbb{R})\times\mathrm{End}(\mathrm{Sym}^p(\mathbb{C}^2))\big) \arrow[r,""swap]\&	U(X)=\Gamma\backslash \mathrm{SL}_2(\mathbb{R})\arrow[dd, bend left=75, "\pi"] \\
			\mathscr{F}=\Gamma\backslash \big({\mathbb{H}^2}\times{C}^\infty(\mathbb{S}^3)\big)	\arrow[rd,""]\arrow[d,"T_{\cdot,p}"swap]  \&	\mathscr{M}=\Gamma\backslash({\mathbb{H}^2}\times\mathbb{S}^3)\arrow[d,"q"]  \\
			\mathrm{End}(F_p)=\Gamma\backslash\big({\mathbb{H}^2}\times\mathrm{End}(\mathrm{Sym}^p(\mathbb{C}^2))\big) \arrow[r,""swap]\&	X=\Gamma\backslash{\mathbb{H}^2}   
		\end{tikzcd}.
	\end{equation}

	\subsection{A functional-measure correspondence}\label{s2.6}

	We now explain that Theorem \ref{C9'} is a uniform version of Theorem \ref{A2}.

	For any $u_p\in C^\infty(X,F_p), s_{p}\in C^\infty(U(X),\pi^*F_p)$, we can form functionals
	\begin{equation}\label{2.30}
		\begin{split}
			L_{u_p}&\colon A\in C^\infty(X,\mathrm{End}(F_p))\mapsto\langle A u_p,u_p\rangle_{L^2(X,F_p)},\\
			L_{s_p}&\colon B\in C^\infty(U(X),\pi^*\mathrm{End}(F_p))\mapsto\langle Bs_p,s_p\rangle_{L^2(U(X),\pi^*F_p)}.
		\end{split}
	\end{equation}
	By \eqref{2.5..}, \eqref{2.13'}, \eqref{2.16'}, and \eqref{diag2}, we can view $u_p\in C^\infty(X,\mathscr{F})\cong C^\infty(\mathscr{M}), s_p\in C^\infty(U(X),\pi^*\mathscr{F}_p)\cong C^\infty(\pi^*\mathscr{M})$. Then we can construct measures
	\begin{equation}\label{2.31}
		\begin{split}
			\lv u_p\rv_\mathbb{C}^2dv_{\mathscr{M}}&\colon \mathscr{A}\in C^\infty(\mathscr{M})\mapsto \int_{\mathscr{M}}\mathscr{A}\lv u_p\rv_\mathbb{C}^2dv_{\mathscr{M}},\\
			\lv s_p\rv_\mathbb{C}^2dv_{\pi^*\mathscr{M}}&\colon \mathscr{B}\in C^\infty(\pi^*\mathscr{M})\mapsto \int_{\pi^*\mathscr{M}}\mathscr{B}\lv s_p\rv_\mathbb{C}^2dv_{\pi^*\mathscr{M}}.
		\end{split}
	\end{equation}

	\begin{prop}\label{p2.3}
		The functionals in \eqref{2.30} and the measures in \eqref{2.31} are equivalent, in the sense made precise below.
	\end{prop}

	\begin{pro}
		We restrict attention to the second functional and the second measure. From \eqref{2.28}, \eqref{2.30}, and \eqref{2.31} we get
		\begin{equation}
			\begin{split}
				L_{s_p}(T_{\mathscr{B},p})&=\langle T_{\mathscr{B},p}s_p,s_p\rangle_{L^2(U(X),\pi^*F_p)}\\
				&=\big\langle \mathscr{B}s_p,s_p\big\rangle_{L^2(\pi^*\mathscr{M})}=	\int_{\pi^*\mathscr{M}}\mathscr{B}\lv s_p\rv_\mathbb{C}^2dv_{\pi^*\mathscr{M}}
			\end{split}
		\end{equation}
		
		This identity shows that pairing $\mathscr{B}\in C^\infty(\pi^*\mathscr{M})$ with $\lv s_p\rv_\mathbb{C}^2dv_{\pi^*\mathscr{M}}$, is equivalent to pairing its quantization $T_{\mathscr{B},p}$ with $L_{s_p}$. Conversely, given any $B\in C^\infty(U(X),\pi^*\mathrm{End}(F_p))$, by Proposition \ref{p2.2}, the fibrewise quantization $T_{\cdot,p}$ in \eqref{2.27} and \eqref{diag} are surjective, hence there exists
		$\mathscr{B}\in C^\infty(\pi^*\mathscr{M})$ such that $B=T_{\mathscr{B},p}$. Hence, pairing $B$ with $L_{s_p}$ is equivalent to pairing $\mathscr{B}$ with $\lv s_p\rv_\mathbb{C}^2dv_{\pi^*\mathscr{M}}$.
		
		Therefore, we conclude that $L_{s_p}$ and $\lv s_p\rv_\mathbb{C}^2dv_{\pi^*\mathscr{M}}$ encode the same information.\qed
	\end{pro}

	Theorems \ref{A2} and Theorem \ref{C9'} study the limits of the functionals in \eqref{2.30} and the measures in \eqref{2.31}. We now compare the two desired limits. We define functionals
	\begin{equation}\label{2.33}
		\begin{split}
		&A\in C^\infty(X,\mathrm{End}(F_p))\mapsto\frac{1}{\dim_\mathbb{C} F_p\cdot \mathrm{Vol}(X)}\int_X\mathrm{Tr}^{F_p}[A]dv_X,\\
	&B\in C^\infty(U(X),\pi^*\mathrm{End}(F_p))\mapsto\frac{1}{\dim_\mathbb{C} F_p\cdot \mathrm{Vol}(U(X))}\int_{U(X)}\mathrm{Tr}^{\pi^*F_p}[B]dv_{U(X)},
		\end{split}
	\end{equation}
	and measures
	\begin{equation}\label{2.34}
		\begin{split}
	&\mathscr{A}\in C^\infty(\mathscr{M})\mapsto \frac{1}{\mathrm{Vol}(\mathscr{M})}\int_{\mathscr{M}}\mathscr{A}dv_{\mathscr{M}}\in \mathbb{C},\\
&\mathscr{B}\in C^\infty(\pi^*\mathscr{M})\mapsto \frac{1}{\mathrm{Vol}(\pi^*\mathscr{M})}\int_{\pi^*\mathscr{M}}\mathscr{B}dv_{\pi^*\mathscr{M}}\in \mathbb{C}.
		\end{split}
	\end{equation}

	\begin{prop}\label{p2.4}
		The functionals in \eqref{2.33} and the measures in \eqref{2.34} are also equivalent.
	\end{prop}
	
	\begin{pro}
		By the trace formula \eqref{2.20}, we have
		\begin{equation}
			\begin{split}
		&\frac{1}{\dim_\mathbb{C} F_p\cdot \mathrm{Vol}(U(X))}\int_{U(X)}\mathrm{Tr}^{\pi^*F_p}[T_{\mathscr{B},p}]dv_{U(X)}\\
				&=\frac{1}{\mathrm{Vol}(\mathbb{S}^3)\cdot \mathrm{Vol}(U(X))}\int_{U(X)}\int_{\mathbb{S}^3}\mathscr{B}(y,z)dv_{\mathbb{S}^3}(z)dv_{U(X)}(y)\\
				&=\frac{1}{\mathrm{Vol}(\pi^*\mathscr{M})}\int_{\pi^*\mathscr{M}}\mathscr{B}dv_{\pi^*\mathscr{M}},
			\end{split}
		\end{equation}
		and combining this with an argument similar to the proof of Proposition \ref{p2.3} yields the equivalence.\qed
	\end{pro}

	\subsection{Uniformity}\label{s2.7}

	By Propositions \ref{p2.3} and \ref{p2.4}, we can see that if we fix a $p\in\mathbb{N}$, then \eqref{9.} is equivalent to \eqref{a6}. This confirms our original intent, and \eqref{9.} is a uniform version of \eqref{a6}.

	Based on the above, it is convenient to formulate the arguments on the infinite dimensional flat vector bundle $\mathscr{F}$. Let $u_{\lambda}\in{C}^\infty(X,\mathscr{F})$ be a normalized eigensection of $\Delta^{\mathscr{F}}$, called a \emph{Maass form}, which satisfies
	\begin{equation}\label{2.36}
		\Delta^{\mathscr{F}}u_{\lambda}=\lambda u_{\lambda},\quad\lV u_{\lambda}\rV_{L^2(X,\mathscr{F})}=1.
	\end{equation}
	We shall study the equidistribution properties of eigensections $u_\lambda$ when $\lambda\to\infty$. A key advantage of this approach is that all subsequent estimates involve constants independent of $p\in\mathbb{N}$, allowing us to take limits uniformly in $p$.

	\section{Arithmetic Uniform QUE}\label{S3AU}
	
	In this section, we state and prove strengthened versions of Theorems \ref{A2} and \ref{C9'}. In \cref{s3.1}, we recall Lindenstrauss's measure rigidity theorem and derive an equidistribution statement on $U(X)$. In \cref{s3.2}, we use the denseness property of the representation \eqref{1.1..} to upgrade this equidistribution to the total space $\pi^*\mathscr{M}$.	In \cref{s3.3'}, we prove our main results. In \cref{s3.3}, we discuss related work and connections to our results.

	We retain the notational conventions from \cref{s1.6} and the geometric setup in \eqref{diag1'}, \eqref{diag2}, and \eqref{diag}.

	\subsection{Measure rigidity}\label{s3.1}

	The following seminal result is due to Lindenstrauss {\cite[Theorem 1.1]{MR2195133}}.
	\begin{theo}\label{t5.1}
Recall that the lattice $\Gamma$ in \eqref{1.1''} is a congruence subgroup of the quaternion algebra $D_{a,b}(L)$. Let $\mu_{U(X)}$ be a probability measure on $U(X)$ with the following properties,
		\begin{itemize}
			\item[{[IG]}] $\mu_{U(X)}$ is invariant under the geodesic flow $(g_t)$ given in \eqref{2.12},
			\item[{[R]}] $\mu_{U(X)}$ is Hecke $\wp$-recurrent for a prime $\wp$,
			\item[{[E]}] almost every ergodic component of $\mu_{U(X)}$ for the geodesic flow $(g_t)$ defined in \eqref{2.12} has positive entropy.
		\end{itemize}
		Then $\mu_{U(X)}$ coincides with the normalized Liouville measure $\frac{1}{\mathrm{Vol}(U(X))}dv_{U(X)}$.
	\end{theo}

	The property [IG] is relatively straightforward, whereas properties [R] and [E] require more detailed definitions and discussion. These will be deferred to \cref{D} and \cref{S7}.

	In \cref{B}, for any Maass form $u_{\lambda}\in{C}^\infty(X,\mathscr{F})$, we will construct a lifted section $s_{\lambda}\in{C}^\infty(U(X),\pi^*\mathscr{F})$ and prove the following result there.
	\begin{prop}\label{B1'}
		Let $\{u_{\lambda_i}\in C^\infty(X,\mathscr{F})\}_{i\in\mathbb{N}}$ be a sequence of Maass forms with lifts $\{s_{\lambda_i}\in {C}^\infty(U(X),\pi^*\mathscr{F})\}_{i\in\mathbb{N}}$. 		Suppose further that $\lim_i \lambda_i=\infty$, and measures $\{\lv u_{\lambda_i}\rv^2_{\mathbb{C}}dv_{\mathscr{M}}\}_{i\in\mathbb{N}}$ converge weak star to a measure $\mu_\mathscr{\mathscr{M}}$ on $\mathscr{M}$. Then for any weak star limit $\mu_{\pi^*\mathscr{M}}$ of $\{\lv s_{\lambda_i}\rv^2_{\mathbb{C}}dv_{\pi^*\mathscr{M}}\}_{i\in\mathbb{N}}$ on $\pi^*\mathscr{M}$, it satisfies the following properties		
			\begin{itemize}
				\item  [{[C]}] the lift $s_{\lambda_i}\in{C}^\infty(U(X),\pi^*\mathscr{F})$ is compatible with the embedding \eqref{2.16'} and \eqref{diag2}, that is, if $u_{\lambda_i}\in{C}^\infty(X,F_p)$, then $s_{\lambda_i}\in{C}^\infty(U(X),\pi^*F_p)$,
		\item[{[P]}] $\mu_{\pi^*\mathscr{M}}$ projects to $\mu_{\mathscr{M}}$ via $\pi\colon\pi^*\mathscr{M}\to \mathscr{M}$ defined in \eqref{diag2},
			\item[{[IH]}] $\mu_{\pi^*\mathscr{M}}$ is invariant under the horizontal geodesic flow $(\bm{g}_t)$ given in \eqref{1.26}. 	In particular, since the horizontal geodesic flow $(\bm{g}_t)$ projects to the geodesic flow $(g_t)$ via $q\colon\pi^*\mathscr{M}\to U(X)$ defined in \eqref{diag2}, the projection measure $\mu_{U(X)}=q_*\mu_{\pi^*\mathscr{M}}$ of $\mu_{\pi^*\mathscr{M}}$ on $U(X)$ is invariant under $(g_t)$. In other words, property [IG] in Theorem \ref{t5.1} holds for $\mu_{U(X)}$.
		\end{itemize}
		
	\end{prop}

	In \cref{C}, we will introduce an additional symmetry on $\mathscr{F}$ from arithmetic theory, given by the Hecke operators. From that point on, when a Maass form $u_\lambda$ defined in \eqref{2.36} satisfies these symmetries, we refer to it as a \emph{Hecke-Maass form}. We can also see from the construction of the lift $s_\lambda$ of $u_\lambda$ that $s_\lambda$ inherits the same symmetry. In \cref{D} and \cref{S7}, we will prove the following two results, respectively.
	
	\begin{prop}\label{D1'}
		Let $\{u_{\lambda_i}\in C^\infty(X,\mathscr{F})\}_{i\in\mathbb{N}}$ be a sequence of Hecke-Maass forms with lifts $\{s_{\lambda_i}\in {C}^\infty(U(X),\pi^*\mathscr{F})\}_{i\in\mathbb{N}}$. Then for any weak star limit $\mu_{\pi^*\mathscr{M}}$ of $\{\lv s_{\lambda_i}\rv^2_{\mathbb{C}}dv_{\pi^*\mathscr{M}}\}_{i\in\mathbb{N}}$, its projection $\mu_{U(X)}$ on $U(X)$ is $\wp$-recurrent for any prime $\wp$.
	\end{prop}

	\begin{prop}\label{E5''}
		Let $\{u_{\lambda_i}\in C^\infty(X,\mathscr{F})\}_{i\in\mathbb{N}}$ be a sequence of Hecke-Maass forms with lifts $\{s_{\lambda_i}\in {C}^\infty(U(X),\pi^*\mathscr{F})\}_{i\in\mathbb{N}}$. Then for any weak star limit $\mu_{\pi^*\mathscr{M}}$ of $\{\lv s_{\lambda_i}\rv^2_{\mathbb{C}}dv_{\pi^*\mathscr{M}}\}_{i\in\mathbb{N}}$, the entropy of almost every ergodic component of its projection $\mu_{U(X)}$ on $U(X)$ is positive for the geodesic flow $(g_t)$.
	\end{prop}
	
	Propositions \ref{B1'}, \ref{D1'}, and \ref{E5''} verify the properties [IG], [R] and [E] of Theorem \ref{t5.1} one by one, and together yield the following result.
	\begin{theo}\label{t3.5}
		Let $\{u_{\lambda_i}\in C^\infty(X,\mathscr{F})\}_{i\in\mathbb{N}}$ be a sequence of Hecke-Maass forms with lifts $\{s_{\lambda_i}\in {C}^\infty(U(X),\pi^*\mathscr{F})\}_{i\in\mathbb{N}}$. Assume further that $\lim_{i}\lambda_i=\infty$, then for any weak star limit $\mu_{\pi^*\mathscr{M}}$ of $\{\lv s_{\lambda_i}\rv^2_{\mathbb{C}}dv_{\pi^*\mathscr{M}}\}_{i\in\mathbb{N}}$, its projection $\mu_{U(X)}$ to $U(X)$ coincides with the normalized Liouville measure $\frac{1}{\mathrm{Vol}(U(X))}dv_{U(X)}$.
	\end{theo}

	By Proposition \ref{p2.3}, which establishes the correspondence between functionals and measures, applying Theorem \ref{t3.5} to a series of Hecke-Maass forms of $\{u_{p,j}\in C^\infty(X,F_p)\}_{i\in\mathbb{N}}$ with fixed $p\in\mathbb{N}$, we obtain the following result.
	\begin{theo}\label{R5.6}
		For a fixed $p\in\mathbb{N}$, let $\{u_{p,j}\in C^\infty(X,F_p)\}_{j\in\mathbb{N}}$ be a sequence of Hecke-Maass forms with lifts $\{s_{p,j}\in {C}^\infty(U(X),\pi^*{F}_p)\}_{j\in\mathbb{N}}$. Assume further that $\lim_{j}\lambda_{p,j}=\infty$, then for any weak star limit $L_{U(X),p}$ of the functionals $\{L_{s_{p,j}}\}_{j\in\mathbb{N}}$ defined in \eqref{2.30}, its restriction to ${C}^\infty(U(X))\cdot \mathrm{Id}(\pi^*F_p)$ also coincides with $\frac{1}{\mathrm{Vol}(U(X))}dv_{U(X)}$.
	\end{theo}

	\subsection{Upgrading equidistribution to the total space}\label{s3.2}

Theorems \ref{t3.5} and \ref{R5.6}, respectively, that the projection of $\mu_{\pi^*\mathscr{M}}$ to $U(X)$ and the restriction of
	$L_{U(X),p}$ to
	$C^\infty(U(X))\cdot\mathrm{Id}_{\pi^*F_p}$
	are both given by the normalized Liouville measure $\frac{1}{\mathrm{Vol}(U(X))}dv_{U(X)}$. However, these are not sufficient to reconstruct $\mu_{\pi^*\mathscr{M}}$ and $L_{U(X),p}$ uniquely.

We therefore make full use of property [IH] in Proposition \ref{B1'}, which is strictly stronger than property [IG]. Another key fact is that the image $\rho(\Gamma)\subset \mathrm{SU}_2(\mathbb{C})$ in \eqref{1.1..} is \emph{dense}. To see the necessity of this denseness, consider the following trivial product
\begin{equation}
U(X)\times \mathbb{S}^3=\big(\Gamma\backslash\mathrm{SL}_2(\mathbb{R})\big)\times \mathbb{S}^3
\end{equation}
obtained by omitting the $\rho$-twist in the definition of $\pi^*\mathscr{M}$ in \eqref{diag2}. Fix a $z\in \mathbb{S}^3$, let $\delta_z$ be the delta measure at $z$, and define product measures
\begin{equation}
\frac{1}{\mathrm{Vol}(U(X))}dv_{U(X)}\delta_z,\quad\frac{1}{\mathrm{Vol}(U(X))\cdot \mathrm{Vol}(\mathbb{S}^3)}dv_{U(X)}dv_{\mathbb{S}^3},
\end{equation}
where $dv_{\mathbb{S}^3}$ is the Euclidean measure on $\mathbb{S}^3$. Both of them are invariant along the horizontal geodesic flow on $U(X)\times \mathbb{S}^3$, and project to $\frac{1}{\mathrm{Vol}(U(X))}dv_{U(X)}$. It follows that the projected measure and horizontal geodesic invariance do not uniquely determine the measure.

	\begin{theo}\label{E5'}
	Let $\mu_{\pi^*\mathscr{M}}$ and $L_{U(X),p}$ be as in Theorems \ref{t3.5} and \ref{R5.6}. Then we have
		\begin{equation}\label{5.17}
			\begin{split}
				\mu_{\pi^*\mathscr{M}}&=\frac{1}{\mathrm{Vol}(\pi^*\mathscr{M})}dv_{\pi^*\mathscr{M}},\\ L_{U(X),p}(\cdot)&=\frac{1}{\dim_{\mathbb{C}}F_p\cdot \mathrm{Vol}(U(X))}\int_{U(X)}\tro^{\pi^*F_p}[\cdot]dv_{U(X)}.
			\end{split}
		\end{equation}
	\end{theo}
	
	\begin{proof}
		By Proposition \ref{p2.4}, the first identity in \eqref{5.17} implies the second. We start with the simpler second identity to explain the argument.
		
		First, we easily check the semi-positivity of $L_{U(X),p}$, that is, for any pointwise positive semi-definite $B\in{C}^\infty(U(X),\pi^*\mathrm{End}({F_p}))$, we have $L_{U(X),p}(B)\geqslant 0$. On the other hand $(\lV B(x)\rV_{\mathrm{End}(F_p)}\cdot\mathrm{Id}_{F_p}-B(x))$ is also positive semi-definite, hence
		\begin{equation}
			\begin{split}
				0\leqslant L_{U(X),p}(B)&\leqslant L_{U(X),p}(\lV B(x)\rV_{\mathrm{End}(F_p)}\cdot\mathrm{Id}_{\pi^*F_p})\\
				&=\frac{1}{\mathrm{Vol}(U(X))}\int_{U(X)}\lV B(x)\rV_{\mathrm{End}(F_p)}dv_{U(X)}(x),
			\end{split}
		\end{equation}
		where the final equality follows from Theorem \ref{R5.6}. This, together with the complex linearity of $L_{U(X),p}$, implies that $L_{U(X),p}$ is \emph{absolutely continuous} with respect to $\frac{1}{\mathrm{Vol}(U(X))}dv_{U(X)}$.

		Using the Riesz representation theorem and lifting $L_{U(X),p}$ to $\mathrm{SL}_2(\mathbb{R})$ via \eqref{diag}, we can write $L_{U(X),p}$ in the form
		\begin{equation}\label{5.19..}
			L_{U(X),p}=\int_{\mathrm{SL}_2(\mathbb{R})}L_{p,y}\cdot dv_{\mathrm{SL}_2(\mathbb{R})}(y),
		\end{equation}
		where $y\in \mathrm{SL}_2(\mathbb{R})\mapsto L_{p,y}$ is a measurable function taking values in $\mathrm{End}\big(\mathrm{Sym}^p(\mathbb{C}^2)\big)^*$, the dual of $\mathrm{End}\big(\mathrm{Sym}^p(\mathbb{C}^2)\big)$. Clearly $L_p$ is $\Gamma$-invariant, namely that for any $x\in {\mathbb{H}^2}, \gamma\in\Gamma$ and $b\in \mathrm{End}\big(\mathrm{Sym}^p(\mathbb{C}^2)\big)$, we have 
		\begin{equation}\label{1.21}
			\big\langle L_{p,\gamma\cdot y},\rho(\gamma)b\rho(\gamma)^{-1}\big\rangle=\langle L_{p,y},b\rangle.
		\end{equation}
		On the other hand, by Hopf argument, the geodesic invariance of $L_{U(X),p}$ implies that $y\mapsto L_{p,y}$ is almost everywhere \emph{constant}. We emphasize that this step relies essentially on property [IH] in Proposition \ref{B1'}. According to the Borel density theorem \cite[Corollary 4.5.6]{MR3307755}, the image $\rho(\Gamma)\subset\mathrm{SU}_2(\mathbb{C})$ in \eqref{1.1..} is \emph{dense}, and so, by \eqref{1.21}, $L_p$ is $\mathrm{SU}_2(\mathbb{C})$-invariant. Using Schur's lemma, for any $b\in \mathrm{End}\big(\mathrm{Sym}^p(\mathbb{C}^2)\big)$, we have
		\begin{equation}
			\frac{1}{\mathrm{Vol}(\mathrm{SU}_2(\mathbb{C}))}\int_{\mathrm{SU}_2(\mathbb{C})}gbg^{-1}dv_{\mathrm{SU}_2(\mathbb{C})}(g)=\frac{1}{\dim_{\mathbb{C}}\mathrm{Sym}^p(\mathbb{C}^2)}\tro^{\mathrm{Sym}^p(\mathbb{C}^2)}[b]\cdot\mathrm{Id}_{\mathrm{Sym}^p(\mathbb{C}^2)},
		\end{equation}
		hence, for a constant $c$,
		\begin{equation}\label{6.6...}
			L_p=c\cdot\tro^{\mathrm{Sym}^p(\mathbb{C}^2)}.
		\end{equation} 
		Finally, since $L_{U(X),p}$ is normalized, that is $L_{U(X),p}(\mathrm{Id}_{\pi^*F_p})=1$, we deduce the second identity in \eqref{5.17} by \eqref{5.19..} and \eqref{6.6...}.
		
		Likewise, we can regard $\mu_{\pi^*\mathscr{M}}$ as a $\Gamma$-invariant measure on $\mathrm{SL}_2(\mathbb{R})\times \mathbb{S}^3$ by \eqref{diag}. Using Theorem \ref{t3.5} and disintegration, we then express $\mu_{\pi^*\mathscr{M}}$ by
		\begin{equation}\label{5.22..}
			\mu_{\pi^*\mathscr{M}}=\int_{\mathrm{SL}_2(\mathbb{R})}\mu_ydv_{\mathrm{SL}_2(\mathbb{R})}(y),
		\end{equation}
		where $\mu_y$ is a measure on $(y,\mathbb{S}^3)\subset \mathrm{SL}_2(\mathbb{R})\times \mathbb{S}^3$. Again, Hopf argument and property [IH] of  Proposition \ref{B1'} show that $y\mapsto\mu_y$ is a almost everywhere constant measure-valued function. Then $\mu$ is a $\mathrm{SU}_2(\mathbb{C})$-invariant measure on $\mathbb{S}^3$, by the Haar theorem \cite[Proposition 4.2.4]{MR3558990}, we have
		\begin{equation}
			\mu=c\cdot dv_{\mathbb{S}^3}.
		\end{equation}
		Since $\mu_{\pi^*\mathscr{M}}$ is also normalized, we get the first identity in \eqref{5.17} from \eqref{5.22..}.
	\end{proof}

\subsection{AUQUE}\label{s3.3'}

	Now we state and prove the full version of our main results.
	
	\begin{theo}\label{A2'}
		For any fixed $p\in\mathbb{N}$ and $B\in{C}^\infty\big(U(X),\pi^*\mathrm{End}(F_p)\big)$, we have
		\begin{equation}\label{e3.12}
			\begin{split}
				\lim_{j\to\infty}L_{s_{p,j}}(B)&=\lim_{j\to\infty}\langle Bs_{p,j},s_{p,j}\rangle_{L^2(U(X),\pi^*F_p)}\\
				&=\frac{1}{\dim_{\mathbb{C}}F_p\cdot \mathrm{Vol}(U(X))}{\int_{U(X)}\tro^{\pi^*F_p}[B]dv_{U(X)}}.
			\end{split}
		\end{equation}
	\end{theo}
	
\begin{pro}
	We argue by contradiction. Suppose that the limit \eqref{e3.12} does not hold. Then there exists a series $(j_i)_{i\in\mathbb{N}}\in\mathbb{N}$ with $\lim_{i\to\infty}j_i=\infty$ such that $\lim_{i\to\infty} L_{s_{p,j_i}}(B)$ exists and
	\begin{equation}\label{3.122}
		\lim_{i\to\infty} L_{s_{p,j_i}}(B)\neq\frac{1}{\dim_{\mathbb{C}}F_p\cdot \mathrm{Vol}(U(X))}{\int_{U(X)}\tro^{\pi^*F_p}[B]dv_{U(X)}}.
	\end{equation}
	However, we can apply Theorems \ref{R5.6} and \ref{E5'} to $\{L_{s_{p,j_i}}\}_{i\in\mathbb{N}}$ to see that \eqref{3.122} must be an equality for a subsequence, yielding a contradiction.\qed
\end{pro}

	By Proposition \ref{B1'}, since $\lim_{j\to\infty}\lambda_{p,j}=\infty$, the weak star limit measure $\lim_{j\to\infty}\lv s_{p,j}\rv^2_{\mathbb{C}}dv_{\pi^*\mathscr{M}}$ projects to the weak star limit measure $\lim_{j\to\infty}\lv u_{p,j}\rv^2_{\mathbb{C}}dv_{\mathscr{M}}$. Then from Proposition \ref{p2.3}, the functional-measure correspondence, the restriction of $\lim_{j\to\infty}L_{s_{p,j}}$ to ${C}^\infty(X,\mathrm{End}(F_p))$ coincides with $\lim_{j\to\infty}L_{u_{p,j}}$. Consequently, Theorem \ref{A2} follows from Theorem \ref{A2'}.

	\begin{theo}\label{E6''}
		For any $\mathscr{B}\in{C}^\infty(\pi^*\mathscr{M})$, we have
		\begin{equation}\label{3.11}
			\lim_{\lambda\to\infty}\sup_{\substack{(p,j)\in\mathbb{N}^{2},\\
					\lambda_{p,j}\geqslant \lambda}}\lv\int_{\pi^*\mathscr{M}}\mathscr{B}\lv s_{p,j}\rv_\mathbb{C}^2d{v}_{\pi^*\mathscr{M}}-\frac{1}{\mathrm{Vol}(\pi^*\mathscr{M})}\int_{\pi^*\mathscr{M}}\mathscr{B}dv_{\pi^*\mathscr{M}}\rv=0.
		\end{equation}
	\end{theo}
	
	\begin{pro}
	Similar to the proof of Theorem \ref{A2'}, we prove by contradiction. Assume that the uniform limit \eqref{3.11} does not hold. Then there exists a series $(p_i,j_i)_{i\in\mathbb{N}}\in\mathbb{N}^2$ with $\lim_{i\to\infty}\lambda_{p_i,j_i}=\infty$ such that
		\begin{equation}\label{3.12'}
			\lim_{i\to\infty}\int_{\pi^*\mathscr{M}}\mathscr{B}\lv s_{p_i,j_i}\rv_\mathbb{C}^2d{v}_{\pi^*\mathscr{M}}\neq\frac{1}{\mathrm{Vol}(\pi^*\mathscr{M})}\int_{\pi^*\mathscr{M}}\mathscr{B}dv_{\pi^*\mathscr{M}}.
		\end{equation}
		On the other hand, by the embedding \eqref{diag2} and the compatibility property [{C}] in Proposition \ref{B1'}, we may apply Theorems \ref{t3.5} and \ref{E5'} to $\{\lv s_{p_i,j_i}\rv_\mathbb{C}^2d{v}_{\pi^*\mathscr{M}}\}_{i\in\mathbb{N}}$ and obtain a contradiction.\qed
	\end{pro}
	
Theorem \ref{C9'} follows from Theorem \ref{E6''} by restricting to $\mathscr{A}\in C^\infty(\mathscr{M})$. Also, the proof of Theorem \ref{E6''} again explains the advantage of working with $\mathscr{F}$, namely that it produces results uniform for all $\{F_p\}_{p\in\mathbb{N}}$.
	
	\subsection{More related results}\label{s3.3}

	Theorem \ref{C9'} is similar in form to Brooks-Lindenstrauss \cite[Theorem 1.5]{MR3260861} and Shem-Tov-Silberman \cite[Theorem 1]{MR5036694}. To explain this similarity, consider a compact quotient of products of hyperbolic manifolds
	\begin{equation}
		\Gamma_{m,n}\backslash\big(\underbrace{\mathbb{H}^2\times\cdots\times\mathbb{H}^2}_{m}\times \underbrace{\mathbb{H}^3\times\cdots\times\mathbb{H}^3}_{n}\big) 
	\end{equation}
	where $\Gamma_{m,n}\subset\mathrm{SL}_2(\mathbb{R})^m\times \mathrm{SL}_2(\mathbb{C})^n$ is a lattice. Then there exists a family of $(m+n)$ \emph{partial Laplacian} operators, one for each factor. Let us consider the joint eigenfunctions of those operators. In the case $m=2,n=0$ with an \emph{irreducible} lattice $\Gamma_{2,0}$, not even necessarily arithmetic, if the eigenvalues tend to infinity in one factor but remain bounded in the other, \cite[Theorem 1.5]{MR3260861} established the QUE. In the case $(m+n)\geqslant1$ and $\Gamma_{m,n}$ is a congruence lattice over a general number field, if the eigenvalues tend to infinity in one factor, \cite[Theorem 1]{MR5036694} asserts the AQUE.
	
	From the above viewpoint, Theorem \ref{C9'} can be interpreted as the AQUE of the partial Laplacian along the hyperbolic component of $\mathscr{M}$, as defined in \eqref{1.10}.
	
	In general, let $G$ be a reductive Lie group, $K$ its maximal compact subgroup and $U$ its compact form. For a lattice $\Gamma_{G,U}\subset G\times U$, let us consider the partial Laplacians on $\Gamma_{G,U}\backslash(G/K\times U)$. If the eigenvalues tend to infinity in one of the hyperbolic factors, we can expect an AQUE as in the special case Theorem \ref{C9'} of $(G,K,U)=(\mathrm{SL}_2(\mathbb{R}),\mathrm{SO}_2(\mathbb{R}),\mathrm{SU}_2(\mathbb{C}))$. Moreover, by restricting to a single irreducible representation of $U$, we may have an AQUE analogous to Theorem \ref{A2}.

	\section{Microlocal Lifts of Eigensections}\label{B}

			In this section, we prove Proposition \ref{B1'}. In \cref{Ba}, we introduce the universal enveloping algebra $U(\mathfrak{sl}_2)$ of $\mathfrak{sl}_2$ and its realization as differential operators. In \cref{s4.2}, we construct microlocal lifts for Maass forms and show property $\mathrm{[C]}$ in Proposition \ref{B1'}. In \cref{s4.3}, we establish property $\mathrm{[P]}$ in Proposition \ref{B1'}. In \cref{s4.4}, we prove property $\mathrm{[IH]}$ in Proposition \ref{B1'}.
	
	We keep the notational conventions introduced in \cref{s1.6} and the geometric setup of \eqref{diag1'}, \eqref{diag2}, and \eqref{diag}.
	
	\subsection{The universal enveloping algebra}\label{Ba}

The Lie algebra $\mathfrak{sl}_2(\mathbb{R})$ of $\mathrm{SL}_2(\mathbb{R})$ acts on $C^\infty(\pi^*\mathscr{M})$ as a first order differential operator by infinitesimal right translations, for $D\in \mathfrak{sl}_2(\mathbb{R})$ and $s\in C^\infty(\pi^*\mathscr{M})$,
\begin{equation}\label{4.1}
(Ds)(y,z)=\frac{d}{dt}\Big|_{t=0}s(ye^{tD},z),
	\end{equation}
where $(y,z)\in \mathrm{SL}_2(\mathbb{R})\times\mathbb{S}^3$. By \eqref{2.13'}, we see that $\mathfrak{sl}_2(\mathbb{R})$ acts equivalently on ${C}^\infty(U(X),\pi^*\mathscr{F})$. Let $\mathfrak{sl}_2(\mathbb{C})=\mathfrak{sl}_2(\mathbb{R})\otimes_\mathbb{R}\mathbb{C}$, the complexification of $\mathfrak{sl}_2(\mathbb{R})$. Then it follows from \eqref{4.1} that $\mathfrak{sl}_2(\mathbb{C})$ acts on $C^\infty(\pi^*\mathscr{M})\cong C^\infty(U(X),\pi^*\mathscr{F})$ as first order differential operators with complex coefficients. Moreover, this action extends naturally to the universal enveloping algebra $U(\mathfrak{sl}_2(\mathbb{C}))$ as higher order differential operators with complex coefficients.

For $D\in\mathfrak{sl}_2(\mathbb{C})$, by Stokes' Theorem, we have for any $s,s'\in {C}^\infty(U(X),\pi^*\mathscr{F})$, 
\begin{equation}
\int_{\pi^*\mathscr{M}}D(s\overline{s'})=0.
\end{equation}
Expanding the integrand yields
\begin{equation}\label{4.3'}
\langle Ds,s'\rangle_{L^2(U(X),\pi^*\mathscr{F})}+\langle s,\overline{D}s'\rangle_{L^2(U(X),\pi^*\mathscr{F})}=0.
\end{equation}
In other words, the formal adjoint of $D$ is $-\overline{D}$.

We now list several operators that will be used later. Define
\begin{equation}\label{3.1}
	\begin{split}
	E^+&=\frac{1}{2}\left(\begin{matrix}1&i\\i&-1\end{matrix}\right),\quad E^-=\frac{1}{2}\left(\begin{matrix}1&-i\\-i&-1\end{matrix}\right),\\
H&=\frac{1}{2}\left(\begin{matrix}1&0\\0&-1\end{matrix}\right),\quad W=\frac{1}{2}\left(\begin{matrix}0&1\\-1&0
\end{matrix}\right),\quad V=\frac{1}{2}\left(\begin{matrix}0&1\\1&0
\end{matrix}\right).
	\end{split}
\end{equation}
Here $E^+$ and $E^-$ are called raising and lowering operators, respectively, $H$ generates the geodesic flow $(g_t)$ defined in \eqref{2.12}, and $W$ generates the maximal compact subgroup $K=\mathrm{SO}_2(\mathbb{R})\subset \mathrm{SL}_2(\mathbb{R})$.

Let $\Omega$ denote the Casimir element
	\begin{equation}\label{3.3}
		\Omega=-\frac{1}{2}\big(E^+E^-+E^-E^+\big)+W^2,
	\end{equation}
	which lies in the center of $U(\mathfrak{sl}_2(\mathbb{C}))$. We also have alternative expressions
	\begin{equation}\label{3.3'}
\Omega=-E^-E^++W^2+iW=-E^+E^-+W^2-iW.
	\end{equation}
	To check this, by \eqref{3.3}, it suffices to verify the second equality. Using \eqref{3.1}, we compute
	\begin{equation}
E^+E^--E^-E^+=(H+iV)(H-iV)-(H-iV)(H+iV)=2i[V,H]=-2iW,
	\end{equation}
as claimed.

Recall the Laplacian $\Delta^{\mathscr{F}}$ defined in \eqref{2.14}, we have on the subspace $C^\infty(X,\mathscr{F})\subset C^\infty(U(X),\pi^*\mathscr{F})$, 
\begin{equation}\label{4.6}
	\Omega=\Delta^{\mathscr{F}}.
\end{equation}

	\subsection{Microlocal lifts}\label{s4.2}

	Recall the Maass form \eqref{2.36}. Our goal is to construct a lifted section $s_{\lambda}\in{C}^\infty(U(X),\pi^*\mathscr{F})$ for each $u_{\lambda}\in{C}^\infty(X,\mathscr{F})$ such that Proposition \ref{B1'} holds.

	For any $n\in\mathbb{Z}$, let $B_{n}$ be the space of  $K$-eigenfunctions of weight $n$,
	\begin{equation}\label{4.11'}
		\begin{split}
			\mathscr{B}_{n}=\bigg\{\mathscr{B}\in{C}^\infty(\pi^*\mathscr{M})\mid& \mathscr{B}(yk_\theta,z)=e^{in\theta}\mathscr{B}(y,z)\\
			&\text{for any}\ (y,z)\in \pi^*\mathscr{M}, k_\theta=\left(\begin{matrix}\cos \theta&\sin \theta\\ -\sin\theta&\cos\theta\end{matrix}\right)\in K\bigg\}.
		\end{split}
	\end{equation}
	Denote $\mathscr{B}_n$ by $S_n$ when regard each function as being in ${C}^\infty(U(X),\pi^*\mathscr{F})$ via \eqref{2.13'},
	\begin{equation}\label{4.13}
		S_{n}=\Big\{s\in{C}^\infty(U(X),\pi^*\mathscr{F})\mid s(yk_\theta)=e^{in\theta}s(y)\ \text{for any}\ y\in U(X), k_\theta\in K\Big\}.
	\end{equation}
	Then we have for $n,n'\in\mathbb{Z}$,
	\begin{equation}\label{4.11}
\mathscr{B}_{0}\cong {C}^\infty(\mathscr{M}),\quad S_{0}\cong{C}^\infty(X,\mathscr{F}),\quad\mathscr{B}_{n}\cdot S_{n'}\subseteq S_{n+n'}, 
	\end{equation}
and when $n\neq n'$,
\begin{equation}\label{4.12}
S_{n}\perp_{L^2(U(X),\pi^*\mathscr{F})} S_{n'}.
\end{equation}

By \eqref{3.1}, on the spaces $\mathscr{B}_n$ and $S_n$, we have
\begin{equation}\label{4.16}
W=\frac{n}{2}.
\end{equation}
Applying \eqref{3.1} again, we can check that
\begin{equation}\label{4.17'}
	WE^\pm=[W,E^\pm]+E^\pm W=\pm iE^\pm+E^\pm W
\end{equation}
Then from \eqref{4.16} and \eqref{4.17'}, we get
\begin{equation}\label{4.17}
E^\pm \mathscr{B}_n\subseteq \mathscr{B}_{n\pm 2},\quad E^\pm S_n\subseteq S_{n\pm 2}.
\end{equation}

Since our analysis focuses on the behavior as $\lambda \to \infty$, without loss of generality, we can assume that $\lambda\geqslant\frac{1}{4}$, then we set $r(\lambda)>0$ by
	\begin{equation}\label{4.17''}
		\lambda=r(\lambda)^2+\frac{1}{4}.
	\end{equation}
	We define a series $\{s_{\lambda,2n}\in S_{2n}\}_{n\in\mathbb{Z}}$ inductively by setting
	\begin{equation}\label{3.7}
		\begin{split}
			s_{\lambda,0}&=u_{\lambda},\\
			s_{\lambda,2n+2}&=\Big(ir(\lambda)+\frac{1}{2}+n\Big)^{-1}E^+s_{\lambda,2n},\quad s_{\lambda,2n-2}=\Big(ir(\lambda)+\frac{1}{2}-n\Big)^{-1}E^-s_{\lambda,2n},
		\end{split}
	\end{equation}
	where the normalized constant ensures that
	\begin{equation}
\lV s_{\lambda,2n}\rV_{L^2(U(X),\pi^*\mathscr{F})}=1.
	\end{equation}

	Now we form the lift $s_{\lambda}$ of $u_{\lambda}$ by
	\begin{equation}\label{3.8}
		s_{\lambda}=\frac{1}{\sqrt{2\lceil r(\lambda)^{1/2}\rceil+1}}\sum_{\lv n\rv\leqslant\lceil r(\lambda)^{1/2}\rceil}s_{\lambda,2n}.
	\end{equation}
	By \eqref{3.7} and \eqref{3.8}, the construction of $s_\lambda$ acts only on the $y$-coordinates part of $u_\lambda(y,z)$. Therefore, property $\mathrm{[C]}$ of Proposition \ref{B1'} holds.

	\subsection{Projection property}\label{s4.3}

	We say that $\mathscr{B}$ is $K$-finite if for some $\ell\in\mathbb{N}$,
	\begin{equation}\label{3.6}
		\mathscr{B}\in \sum_{\lv n\rv\leqslant 2\ell}\mathscr{B}_{n}.
	\end{equation}
	
	\begin{prop}
		For any $K$-finite function  $\mathscr{B}\in{C}^\infty(\pi^*\mathscr{M})$, there is $C>0$ such that for any $u_{\lambda}$ and its lifting $s_{\lambda}$, we have
		\begin{equation}\label{3.10}
			\begin{split}
				\bigg\vert \langle \mathscr{B}s_{\lambda},s_{\lambda}\rangle_{L^2(U(X),\pi^*\mathscr{F})}-\sum_{\lv n\rv\leqslant\lceil r(\lambda)^{1/2}\rceil}\langle \mathscr{B}s_{\lambda,2n},u_{\lambda}\rangle_{L^2(U(X),\pi^*\mathscr{F})}\bigg\vert\\
				\leqslant Cr(\lambda)^{-1/2}\lV \mathscr{B}\rV_{C^1(\pi^*\mathscr{M})}.
			\end{split}
		\end{equation}
	\end{prop}
	
	\begin{proof}
		By \eqref{4.3'} and \eqref{3.7}, we get
		\begin{equation}\label{4.15}
			\begin{split}
				\langle \mathscr{B}&s_{\lambda,2m},s_{\lambda,2n}\rangle_{L^2(U(X),\pi^*\mathscr{F})}\\
				=&\Big(ir(\lambda)-\frac{1}{2}+m\Big)^{-1}\langle \mathscr{B}E^+ s_{\lambda,2m-2},s_{\lambda,2n}\rangle_{L^2(U(X),\pi^*\mathscr{F})}\\
				=&\Big(ir(\lambda)-\frac{1}{2}+m\Big)^{-1}\big\langle E^+(\mathscr{B}s_{\lambda,2m-2})-E^+(\mathscr{B})s_{\lambda,2m},s_{\lambda,2n}\big\rangle_{L^2(U(X),\pi^*\mathscr{F})}.
			\end{split}
		\end{equation}
	Applying \eqref{4.3'} and \eqref{3.7} once more, this expression equals
			\begin{equation}\label{4.15'}
			\begin{split}
	-\Big(ir(\lambda)&-\frac{1}{2}+m\Big)^{-1}\big\langle \mathscr{B}s_{\lambda,2m-2},E^-s_{\lambda,2n}\big\rangle_{L^2(U(X),\pi^*\mathscr{F})}\\
	-\Big(ir(\lambda)&-\frac{1}{2}+m\Big)^{-1}\big\langle E^+(\mathscr{B})s_{\lambda,2m},s_{\lambda,2n}\big\rangle_{L^2(U(X),\pi^*\mathscr{F})}\\
				=&\frac{\Big(ir(\lambda)-\frac{1}{2}+n\Big)}{\Big(ir(\lambda)-\frac{1}{2}+m\Big)}\langle \mathscr{B}s_{\lambda,2m-2},s_{\lambda,2n-2}\rangle_{L^2(U(X),\pi^*\mathscr{F})}\\
				&-\Big(ir(\lambda)-\frac{1}{2}+m\Big)^{-1}\big\langle E^+(\mathscr{B})s_{\lambda,2m},s_{\lambda,2n}\big\rangle_{L^2(U(X),\pi^*\mathscr{F})}.
			\end{split}
		\end{equation}
	Combining \eqref{4.15}, \eqref{4.15'}, and using the estimate $\lV E^+(\mathscr{B})\rV_{C^0(\pi^*\mathscr{M})}\leqslant \lV\mathscr{B}\rV_{C^1(\pi^*\mathscr{M})}$, we conclude that
		\begin{equation}\label{3.12}
			\begin{split}
				\bv\langle \mathscr{B}s_{\lambda,2m},s_{\lambda,2n}\rangle_{L^2(U(X),\pi^*\mathscr{F})}&-\langle \mathscr{B}s_{\lambda,2m-2},s_{\lambda,2n-2}\rangle_{L^2(U(X),\pi^*\mathscr{F})}\bv\\
				&\leqslant C\big(\lV\mathscr{B}\rV_{C^1(\pi^*\mathscr{M})}+\lv m-n\rv\lV\mathscr{B}\rV_{C^0(\pi^*\mathscr{M})}\big)r(\lambda)^{-1}.
			\end{split}
		\end{equation}

	By \eqref{4.11}, \eqref{4.12}, and \eqref{3.6}, we see that whenever $\lv n-m\rv>\ell$,
	\begin{equation}\label{4.20}
\langle\mathscr{B}s_{\lambda,2m},s_{\lambda,2n}\rangle_{L^2(U(X),\pi^*\mathscr{F})}=0,
	\end{equation}
so we can replace $\lv m-n\rv$ with $C$ in \eqref{3.12}. Repeating this, by \eqref{3.7} we get
		\begin{equation}\label{3.13}
			\begin{split}
				\bv\langle \mathscr{B} s_{\lambda,2m},s_{\lambda,2n}\rangle_{L^2(U(X),\pi^*\mathscr{F})}-\langle \mathscr{B}s_{\lambda,2(m-n)},u_{\lambda}\rangle_{L^2(U(X),\pi^*\mathscr{F})}\bv\\
				\leqslant C\lV\mathscr{B}\rV_{C^1(\pi^*\mathscr{M})}\lv n\rv r(\lambda)^{-1}
			\end{split}
		\end{equation}
since $\lV\mathscr{B}\rV_{C^0(\pi^*\mathscr{M})}\leqslant\lV\mathscr{B}\rV_{C^1(\pi^*\mathscr{M})}$.
		
		By \eqref{3.8} and \eqref{4.20}, we get
		\begin{equation}\label{4.22}
\begin{split}
\langle \mathscr{B}s_{\lambda},s_{\lambda}\rangle_{L^2(U(X),\pi^*\mathscr{F})}&=\frac{1}{2\lceil r(\lambda)^{1/2}\rceil+1}\sum_{\lv m\rv,\lv n\rv\leqslant\lceil r(\lambda)^{1/2}\rceil}\langle \mathscr{B}s_{\lambda,2m},s_{\lambda,2n}\rangle_{L^2(U(X),\pi^*\mathscr{F})}\\
&=\frac{1}{2\lceil r(\lambda)^{1/2}\rceil+1}\sum_{\substack{\lv m\rv,\lv n\rv\leqslant\lceil r(\lambda)^{1/2}\rceil,\\ \lv n-m\rv\leqslant \ell}}\langle \mathscr{B}s_{\lambda,2m},s_{\lambda,2n}\rangle_{L^2(U(X),\pi^*\mathscr{F})}.
\end{split}
		\end{equation}
We compute that
			\begin{equation}\label{4.23}
	\begin{split}
	\frac{1}{2\lceil r(\lambda)^{1/2}\rceil+1}&\sum_{\substack{\lv m\rv,\lv n\rv\leqslant\lceil r(\lambda)^{1/2}\rceil,\\ \lv n-m\rv\leqslant \ell}}\langle \mathscr{B}s_{\lambda,2(m-n)},u_{\lambda}\rangle_{L^2(U(X),\pi^*\mathscr{F})}\\
=&\sum_{\lv n\rv\leqslant\ell}\frac{2\lceil r(\lambda)^{1/2}\rceil+1-\lv n\rv}{2\lceil r(\lambda)^{1/2}\rceil+1}\langle\mathscr{B}s_{\lambda,2n},u_{\lambda}\rangle_{L^2(U(X),\pi^*\mathscr{F})}.
	\end{split}
\end{equation}
Taken together \eqref{3.13}, \eqref{4.22}, and \eqref{4.23}, we obtain
		\begin{equation}
			\begin{split}
				\bbv\langle \mathscr{B}s_{\lambda},s_{\lambda}\rangle_{L^2(U(X),\pi^*\mathscr{F})}-\sum_{\lv n\rv\leqslant\ell}\frac{2\lceil r(\lambda)^{1/2}\rceil+1-\lv n\rv}{2\lceil r(\lambda)^{1/2}\rceil+1}\langle\mathscr{B}s_{\lambda,2n},u_{\lambda}\rangle_{L^2(U(X),\pi^*\mathscr{F})}\bbv\\
				\leqslant C\lV\mathscr{B}\rV_{C^1(\pi^*\mathscr{M})}r(\lambda)^{-1/2},
			\end{split}
		\end{equation}
		from which we get \eqref{3.10} again by $\lV\mathscr{B}\rV_{C^0(\pi^*\mathscr{M})}\leqslant\lV\mathscr{B}\rV_{C^1(\pi^*\mathscr{M})}$.
	\end{proof}
		
	Applying \eqref{3.10}  to $\mathscr{A}\in{C}^\infty(\mathscr{M})=\mathscr{B}_{0}$ by \eqref{4.11}, since $\mathscr{A}s_{\lambda,2n}\in S_{2n}$ is orthogonal to $u_{\lambda}\in S_{0}$ for $n\neq0$, we obtain
	\begin{equation}
		\bigg\vert\langle \mathscr{A}s_{\lambda},s_{\lambda}\rangle_{L^2(U(X),\pi^*\mathscr{F})}-\langle \mathscr{A}u_{\lambda},u_{\lambda}\rangle_{L^2(X,\mathscr{F})}\bigg\vert\leqslant Cr(\lambda)^{-1/2}\lV \mathscr{A}\rV_{C^1(\mathscr{M})}.
	\end{equation}
Taking $\lambda\to\infty$, this proves property $\mathrm{[P]}$ in Proposition \ref{B1'}.

	\subsection{Horizontal geodesic invariance}\label{s4.4}
	
	\begin{prop}
		For any $K$-finite function  $\mathscr{B}\in{C}^\infty(\pi^*\mathscr{M})$, there is $C>0$ such that for any $u_{\lambda}$, we have
		\begin{equation}\label{3.16}
			\bigg\vert\sum_{\lv n\rv\leqslant\lceil r(\lambda)^{1/2}\rceil}\langle H(\mathscr{B})s_{\lambda,2n},u_{\lambda}\rangle_{L^2(U(X),\pi^*\mathscr{F})}\bigg\vert\leqslant Cr(\lambda)^{-1/2}\lV\mathscr{B}\rV_{C^2(\pi^*\mathscr{M})}.
		\end{equation}
	\end{prop}

	\begin{proof}
		By \eqref{4.11}, we have $u_{\lambda}\in S_{0}$, hence
		\begin{equation}\label{4.30}
Wu_{\lambda}=0.
		\end{equation}		
		Then by \eqref{2.36}, \eqref{3.3'}, and \eqref{4.6}, we have
	\begin{equation}\label{3.18..}
		\begin{split}
			\lambda\langle \mathscr{B}s_{\lambda,2n},u_{\lambda}\rangle_{L^2(U(X),\pi^*\mathscr{F})}&=\langle \mathscr{B}s_{\lambda,2n},\Omega u_{\lambda}\rangle_{L^2(U(X),\pi^*\mathscr{F})}\\
			&=-\langle \mathscr{B}s_{\lambda,2n},E^-E^+u_{\lambda}\rangle_{L^2(U(X),\pi^*\mathscr{F})}.
			\end{split}
		\end{equation}
		Summing over $\lv n\rv\leqslant\lceil r(\lambda)^{1/2}\rceil$, by \eqref{4.3'}, we get
		\begin{equation}\label{3.18.}
			\begin{split}
				\lambda\sum_{\lv n\rv\leqslant\lceil r(\lambda)^{1/2}\rceil}\langle \mathscr{B}s_{\lambda,2n},u_{\lambda}\rangle_{L^2(U(X),\pi^*\mathscr{F})}&=\sum_{\lv n\rv\leqslant\lceil r(\lambda)^{1/2}\rceil}\langle \mathscr{B}s_{\lambda,2n},E^-E^+u_{\lambda}\rangle_{L^2(U(X),\pi^*\mathscr{F})}\\
				&=-\sum_{\lv n\rv\leqslant\lceil r(\lambda)^{1/2}\rceil}\langle E^-E^+( \mathscr{B}s_{\lambda,2n}),u_{\lambda}\rangle_{L^2(U(X),\pi^*\mathscr{F})},
			\end{split}
		\end{equation}
	and the last term further equals
		\begin{equation}\label{3.18}
			\begin{split}
				-\sum_{\lv n\rv\leqslant\lceil r(\lambda)^{1/2}\rceil}&\langle E^-E^+(\mathscr{B})s_{\lambda,2n}+\mathscr{B}E^-E^+(s_{\lambda,2n}),u_{\lambda}\rangle_{L^2(U(X),\pi^*\mathscr{F})}\\
				&+\langle E^-(\mathscr{B})E^+(s_{\lambda,2n})+E^+(\mathscr{B})E^-(s_{\lambda,2n}),u_{\lambda}\rangle_{L^2(U(X),\pi^*\mathscr{F})}.
			\end{split}
		\end{equation}

	To simplify \eqref{3.18}, let us do a little bit preparation. By the construction \eqref{3.7} and the fact that $\Omega$ lies in the center of $U(\mathfrak{sl}_2(\mathbb{C}))$, all $s_{\lambda,2n}$ are also eigensections of $\Omega$ with eigenvalue $\lambda$, that is,
		\begin{equation}\label{4.34}
\Omega s_{\lambda,2n}=\lambda s_{\lambda,2n}.
		\end{equation}
On the other hand, by \eqref{4.30}, for any $s\in C^\infty(U(X),\pi^*\mathscr{F})$, we have
		\begin{equation}\label{4.35}
			\begin{split}
0=-\langle \mathscr{B}s,W(u_{\lambda})\rangle_{L^2(U(X),\pi^*\mathscr{F})}=\langle W(\mathscr{B}s),u_{\lambda}\rangle_{L^2(U(X),\pi^*\mathscr{F})}\\
=\langle W(\mathscr{B})s,u_{\lambda}\rangle_{L^2(U(X),\pi^*\mathscr{F})}+\langle \mathscr{B}W(s),u_{\lambda}\rangle_{L^2(U(X),\pi^*\mathscr{F})}
			\end{split}
		\end{equation}
	With \eqref{4.34} and \eqref{4.35} in hand, we express the second term in \eqref{3.18} as
		\begin{equation}\label{4.40}
	\begin{split}
		-&\sum_{\lv n\rv\leqslant\lceil r(\lambda)^{1/2}\rceil}\langle\mathscr{B}E^-E^+(s_{\lambda,2n}),u_{\lambda}\rangle_{L^2(U(X),\pi^*\mathscr{F})}\\
		=&\sum_{\lv n\rv\leqslant\lceil r(\lambda)^{1/2}\rceil}\big\langle\mathscr{B}\big(\Omega-W^2-iW\big)(s_{\lambda,2n}),u_{\lambda}\big\rangle_{L^2(U(X),\pi^*\mathscr{F})}\\
		=&\lambda\sum_{\lv n\rv\leqslant\lceil r(\lambda)^{1/2}\rceil}\langle\mathscr{B}s_{\lambda,2n},u_{\lambda}\rangle_{L^2(U(X),\pi^*\mathscr{F})}\\
		&+\sum_{\lv n\rv\leqslant\lceil r(\lambda)^{1/2}\rceil}\big\langle \big(-W^2+iW\big)(\mathscr{B})s_{\lambda,2n},u_{\lambda}\big\rangle_{L^2(U(X),\pi^*\mathscr{F})}.
	\end{split}
\end{equation}		
		
Then we turn to the third and the fourth terms in \eqref{3.18}. By \eqref{4.17} and \eqref{3.6}, we get
		\begin{equation}\label{4.41}
E^\pm(\mathscr{B})\in \sum_{\lv n\rv\leqslant 2\ell+2}S_{n}
		\end{equation}
	By \eqref{4.16}, \eqref{4.17}, \eqref{3.7}, and \eqref{3.8}, we have
		\begin{equation}\label{4.42}
\big(\sum_{\lv n\rv\leqslant\lceil r(\lambda)^{1/2}\rceil}E^\pm s_{\lambda,2n}\big)-\big(ir(\lambda)\mp W-\tfrac{1}{2}\big)s_{\lambda}\in S_{\pm(2\lceil r(\lambda)^{1/2}\rceil+2)}+S_{\mp2\lceil r(\lambda)^{1/2}\rceil}
		\end{equation}
Putting together \eqref{4.41} and \eqref{4.42}, by \eqref{4.11} and \eqref{4.12}, we see that for $r(\lambda)$ large enough,
\begin{equation}\label{4.43}
E^{\mp}(\mathscr{B})\Big(\big(\sum_{\lv n\rv\leqslant\lceil r(\lambda)^{1/2}\rceil}E^\pm s_{\lambda,2n}\big)-\big(ir(\lambda)\mp W-\tfrac{1}{2}\big)s_{\lambda}\Big)\perp_{L^2(U(X),\pi^*\mathscr{F})} u_\lambda.
\end{equation}		
	Now by \eqref{3.1}, \eqref{4.35}, and \eqref{4.43}, the third and the fourth terms in \eqref{3.18} can be rewritten as
\begin{equation}\label{4.44}
	\begin{split}
		-\sum_{\lv n\rv\leqslant\lceil r(\lambda)^{1/2}\rceil}\big\langle E^-(\mathscr{B})&\big(ir(\lambda)-W-\tfrac{1}{2}\big)(s_{\lambda,2n}),u_{\lambda}\big\rangle_{L^2(U(X),\pi^*\mathscr{F})}\\
		+\big\langle E^+(&\mathscr{B})\big(ir(\lambda)+W-\tfrac{1}{2}\big)(s_{\lambda,2n}),u_{\lambda}\big\rangle_{L^2(U(X),\pi^*\mathscr{F})}\\
		=-\sum_{\lv n\rv\leqslant\lceil r(\lambda)^{1/2}\rceil}&\big\langle \big(ir(\lambda)-\tfrac{1}{2}\big)(E^++E^-)(\mathscr{B})(s_{\lambda,2n}),u_{\lambda}\big\rangle_{L^2(U(X),\pi^*\mathscr{F})}\\
		&+\big\langle W(E^--E^+)(\mathscr{B})s_{\lambda,2n},u_{\lambda}\big\rangle_{L^2(U(X),\pi^*\mathscr{F})}\\
		=-\sum_{\lv n\rv\leqslant\lceil r(\lambda)^{1/2}\rceil}&\big\langle \big(2ir(\lambda)-1\big)H(\mathscr{B})(s_{\lambda,2n}),u_{\lambda}\big\rangle_{L^2(U(X),\pi^*\mathscr{F})}\\
		&-\big\langle 2iWV(\mathscr{B})s_{\lambda,2n},u_{\lambda}\big\rangle_{L^2(U(X),\pi^*\mathscr{F})}
	\end{split}
\end{equation}

Combining \eqref{3.18}, \eqref{4.40}, and \eqref{4.44}, we get for some $D_2\in U^2(\mathfrak{sl}_2(\mathbb{C}))$,
		\begin{equation}\label{3.19}
			\sum_{\lv n\rv\leqslant\lceil r(\lambda)^{1/2}\rceil}(2ir(\lambda)-1)\langle H(\mathscr{B})s_{\lambda,2n},u_{\lambda}\rangle_{L^2(U(X),\pi^*\mathscr{F})}+\langle D(\mathscr{B})s_{\lambda,2n},u_{\lambda}\rangle_{L^2(U(X),\pi^*\mathscr{F})}=0.
		\end{equation}
	The estimate \eqref{3.16} then follows immediately from \eqref{3.19}.
	\end{proof}

	By \eqref{3.8}, \eqref{3.10}, and \eqref{3.16}, for any $K$-finite $\mathscr{B}\in{C}^\infty(\pi^*\mathscr{M})$, we have
	\begin{equation}\label{3.21}
		\bv\langle H(\mathscr{B})s_{\lambda},s_{\lambda}\rangle_{L^2(U(X),\pi^*\mathscr{F})}\bv\leqslant  C\lV\mathscr{B}\rV_{C^2(\pi^*\mathscr{M})}r(\lambda)^{-1/2}.
	\end{equation}
	This implies property $\mathrm{[IH]}$ in Proposition \ref{B1'} for $K$-finite $\mathscr{B}\in{C}^\infty(\pi^*\mathscr{M})$. This extends to a general $\mathscr{B}\in{C}^\infty(\pi^*\mathscr{M})$ by the partial sum of its Fourier series as a $K$-finite approximation.

	\section{Hecke-Maass Forms and Hecke Operators}\label{C}

	In this section, we discuss additional symmetries on arithmetic surfaces, known as Hecke operators. In \cref{C1}, we define quaternionic arithmetic surfaces and the associated flat vector bundles. In \cref{SCb}, we describe arithmetic surfaces in terms of adelic quotients. In \cref{Cc}, we introduce Hecke operators from both the double coset and adelic perspectives. In \cref{S5.4}, we define Hecke-Maass forms

	We continue to use the notation fixed in \cref{s1.6} and the geometric setup \eqref{diag1'}, \eqref{diag2}, and \eqref{diag}.

	\subsection{Quaternionic arithmetic groups}\label{C1}

	Let $L$ be a \emph{totally real} number field with the ring of integers $\mathcal{O}_L$ and the set of places $\mathrm{Pl}(L)$.

	For any $a,b\in L$, let $D_{a,b}(L)$ denote the associated quaternion algebra over $L$ given by
\begin{equation}
D_{a,b}(L)=\{x_0+x_1\mathrm{i}+x_2\mathrm{j}+x_3\mathrm{ij}\mid x_0,\cdots,x_3\in L,\ \mathrm{i}^2=a, \mathrm{j}^2=b, \mathrm{ij}=-\mathrm{ji}\}
\end{equation}
with reduced norm
\begin{equation}
\mathrm{Nrd}(x_0+x_1\mathrm{i}+x_2\mathrm{j}+x_3\mathrm{ij})=x_0^2-ax_1^2-bx_2^2+abx_3^2.
\end{equation}
For a real  $\tau\in \mathrm{Pl}(L)$, we have
	\begin{equation}\label{4.2''}
		D_{a,b}(L)\otimes_LL_{\tau}\cong D_{\tau(a),\tau(b)}(\mathbb{R})\cong\begin{cases}
			M_{2}(\mathbb{R}), & \text{if not both } \tau(a) \text{ and } \tau(b)<0,\\
			D_{-1,-1}(\mathbb{R}),    & \text{if } \tau(a),\tau(b)<0,
		\end{cases}
	\end{equation}
	where $M_{2}(\mathbb{R})$ denotes the ring of $2\times2$ matrices and $D_{-1,-1}(\mathbb{R})$ denotes Hamilton's quaternion algebra. This isomorphism can be written explicitly as follows
	\begin{equation}\label{4.3''}
		\begin{split}
			&x_0+x_1\mathrm{i}+x_2\mathrm{j}+x_3\mathrm{ij}\in D_{\tau(a),\tau(b)}(\mathbb{R})\\
			&\mapsto \begin{cases}
				\begin{pmatrix}
					x_0+x_1\sqrt{\tau(a)}&\tau(b)(x_2+x_3\sqrt{\tau(a)})\\
					x_2-x_3\sqrt{\tau(a)}&x_0-x_1\sqrt{\tau(a)}
				\end{pmatrix}\in M_{2}(\mathbb{R}), &\text{ if } \tau(a)>0,\\
				x_0+\sqrt{-\tau(a)}x_1\mathrm{i}+\sqrt{-\tau(b)}x_2\mathrm{j}+\sqrt{\tau(ab)}x_3\mathrm{ij}\in D_{-1,-1}(\mathbb{R}), &\text{ if } \tau(a),\tau(b)<0.
			\end{cases}
		\end{split}
	\end{equation}

	Now we suppose further that
	\begin{equation}\label{4.4'}
			a,b>0,\quad\tau(a),\tau(b)<0 \text{ for any nonidentity real } \tau\in \mathrm{Pl}(L).
	\end{equation}
A simple example is $L=\mathbb{Q}[\sqrt{2}], a=b=\sqrt{2}$. We take 
\begin{equation}\label{5.5.}
\mathbb{G}=\text{the algebraic group of elements in } D_{a,b}(L) \text{ with }\mathrm{Nrd}=1,
\end{equation}
then $\mathbb{G}(\mathcal{O}_L)\subset \mathrm{SL}_2(\mathbb{R})$ is cocompact, see Morris \cite[Proposition 6.2.5]{MR3307755}. We assume that
	\begin{equation}\label{4.5'}
		\Gamma\subset\mathbb{G}(\mathcal{O}_L)
	\end{equation}
	is a \emph{congruence subgroup}.
	
	Let us now fix a \emph{nonidentity real} $\rho\in \mathrm{Pl}(L)$, then it extends to a representation
	\begin{equation}\label{4.2'}
		\rho\colon \mathbb{G}(L)\to\mathrm{SU}_2(\mathbb{C})
	\end{equation}
	by \eqref{4.2''}, \eqref{4.3''}, and the fact that reduced norm one Hamilton's quaternions are isomorphic to $\mathrm{SU}_2(\mathbb{C})$. By composing with \eqref{1.8'}, we see that $\mathbb{G}(L)$ acts unitarily on ${C}^\infty(\mathbb{S}^3)$. 
	
	Now, we suppose that the bundles $(F_p,\mathscr{F},\pi^*F_p,\pi^*\mathscr{F})$ given in \eqref{diag2} and \eqref{diag} are defined by the restriction of $\rho$ to the subgroup $\Gamma\subseteq \mathbb{G}(\mathcal{O}_L)\subset \mathbb{G}(L)$. This is the key point that allows us to define Hecke operators.
	
	\subsection{Adelic quotients}\label{SCb}
	
	For any \emph{finite} $\wp\in\mathrm{Pl}(L)$, let $L_\wp$ be the completion of $L$ at $\wp$ and  $\mathcal{O}_\wp\subset L_\wp$ the maximal compact subring. Let $\pi_\wp$ be a \emph{uniformizer} of $L_\wp$ and $\ell_\wp=\mathcal{O}_\wp/\pi_\wp\mathcal{O}_\wp$ the residue field of $\mathcal{O}_\wp$, then we put $q_\wp=\lv \ell_\wp\rv$. For \emph{all but finite} places $\wp$, we have the inclusion of a maximal compact subgroup
	\begin{equation}
		\big(\mathbb{G}(\mathcal{O}_\wp)\cong\mathrm{SL}_2(\mathcal{O}_\wp)\big)\hookrightarrow \big(\mathbb{G}(L_\wp)\cong\mathrm{SL}_2(L_\wp)\big),
	\end{equation}
	and in what follows, we restrict to such good finite places. Put
	\begin{equation}
		\mathbb{G}(\mathbb{A}_f)={\textstyle\prod_{v<\infty}^{'}}\mathbb{G}(L_\wp),
	\end{equation}
	the \emph{restricted direct product} of $\mathbb{G}(L_\wp)$ relative to $\mathbb{G}(\mathcal{O}_\wp)$, and we set
	\begin{equation}
		\mathbb{G}(\mathbb{A}_L)=\mathrm{SL}_2(\mathbb{R})\times\mathbb{G}(\mathbb{A}_f).
	\end{equation}

	By the strong approximation theorem, see Voight \cite[Corollary 28.6.8]{MR4279905}, let $K_f\subset \mathbb{G}(\mathbb{A}_f)$ be an open compact subgroup such that $\Gamma=\mathbb{G}(L)\cap K_f$, then we have an isomorphism 
	\begin{equation}\label{4.2}
		\begin{split}
	U(X)\cong\Gamma\backslash \mathrm{SL}_2(\mathbb{R})&\cong\mathbb{G}(L)\backslash\mathbb{G}(\mathbb{A}_L)/K_f,\\
[g_\infty]=\Gamma(g_\infty)&\mapsto [(g_\infty,1)]=\mathbb{G}(L)(g_\infty,1)K_f,
		\end{split}
	\end{equation}
where $\mathbb{G}(L)$ acts diagonally on $\mathbb{G}(\mathbb{A}_L)$. Applying the strong approximation again, for any $g_f\in\mathbb{G}(\mathbb{A}_f)$, there exist $g_L\in \mathbb{G}(L), k_f\in K_f$ such that
\begin{equation}
g_f=g_Lk_f,
\end{equation}
therefore, the inverse of \eqref{4.2} is given by
\begin{equation}\label{5.11}
[(g_\infty,g_f)]=[(g_\infty,g_Lk_f)]=[(g_L^{-1}g_\infty,1)]\mapsto [g_L^{-1}g_\infty].
\end{equation}
From \eqref{4.2}, denoting by $C_{\mathbb{G}(L)}^\infty\big(\mathbb{G}(\mathbb{A}_L)/K_f\big)$ the $\mathbb{G}(L)$-invariant subspace of ${C}^\infty(\mathbb{G}(\mathbb{A}_L)/K_f)$, then we have
	\begin{equation}\label{4.3}
		{C}^\infty\big(U(X)\big)\cong{C}^\infty_\Gamma\big(\mathrm{SL}_2(\mathbb{R})\big)\cong C_{\mathbb{G}(L)}^\infty\big(\mathbb{G}(\mathbb{A}_L)/K_f\big).
	\end{equation}
	Similarly, taking into account the $\mathbb{G}(L)$-action on ${C}^\infty(\mathbb{S}^3)$ given in \eqref{4.2'}, we see that
	\begin{equation}\label{4.4}
		\begin{split}
	\pi^*\mathscr{F}=\Gamma\backslash \big(\mathrm{SL}_2(\mathbb{R})\times{C}^\infty(\mathbb{S}^3)\big)&\cong \mathbb{G}(L)\backslash\big(\mathbb{G}(\mathbb{A}_L)\times{C}^\infty(\mathbb{S}^3)\big)/K_f,\\
[(g_\infty,\alpha)]=\Gamma(g_\infty,\alpha)&\mapsto [(g_\infty,1),\alpha]=\mathbb{G}(L)\big((g_\infty,1),\alpha\big)K_f,
		\end{split}
	\end{equation}
	where $\mathbb{G}(L)$ again acts diagonally. Moreover, ${C}^\infty(U(X),\pi^*\mathscr{F})$ is isomorphic to the $\mathbb{G}(L)$-invariant subspace $C_{\mathbb{G}(L)}^\infty\big(\mathbb{G}(\mathbb{A}_L)/K_f,{C}^\infty(\mathbb{S}^3)\big)\hookrightarrow{C}^\infty\big(\mathbb{G}(\mathbb{A}_L)/K_f,{C}^\infty(\mathbb{S}^3)\big)$, namely
	\begin{equation}\label{4.5}
		{C}^\infty\big(U(X),\pi^*\mathscr{F}\big)\cong C_\Gamma^\infty\big(\mathrm{SL}_2(\mathbb{R}),{C}^\infty(\mathbb{S}^3)\big)\cong C_{\mathbb{G}(L)}^\infty\big(\mathbb{G}(\mathbb{A}_L)/K_f,{C}^\infty(\mathbb{S}^3)\big).
	\end{equation}

	\subsection{Hecke operators}\label{Cc}

	First, we discuss Hecke operators using the first isomorphism in \eqref{4.3} and \eqref{4.5}.
	
	\subsubsection{Double cosets}
	
	For any $\gamma\in \mathbb{G}(L)$, the set $[\gamma]=\Gamma \backslash\Gamma \gamma\Gamma$ is finite, and we can define the corresponding \emph{Hecke operator} $T_{\gamma}$ acting on $C^\infty(U(X),\pi^*\mathscr{F})$ such that for any $s\in C^\infty(U(X),\pi^*\mathscr{F})$,
	\begin{equation}\label{4.8}
		T_{\gamma}(s)(y)=\sum_{\gamma'\in[\gamma]}\rho(\gamma')^{-1}s(\gamma' y),
	\end{equation}
	where we view $s\in{C}^\infty\big(\mathrm{SL}_2(\mathbb{R}),{C}^\infty(\mathbb{S}^3)\big)$ through \eqref{4.5}. Also, $T_{\gamma}$ restricts to operators acting on subspaces
	\begin{equation}\label{4.12''}
		C^\infty(X,\mathscr{F}), C^\infty(U(X)), C^\infty(X)\hookrightarrow C^\infty(U(X),\pi^*\mathscr{F}).
	\end{equation}

	In \eqref{4.8}, the operator $\rho(\gamma')$ is unitary, therefore, as we will see in the next section, estimates for eigenfunctions of $T_\gamma|_{{C}^\infty(U(X))}$ in \cite[\S\,8]{MR2195133}, \cite[\S\,5]{MR5036694} and \cite[\S\,3, \S\,5]{MR4033919} also hold for eigensections of $T_{\gamma}$, confirming again the uniformity discussed in \cref{s2.7}.

	Now, we turn to a different adelic perspective of these operators using the second isomorphism in \eqref{4.5}.
	
	\subsubsection{Adelic convolutions}

	For any \emph{finite} $\wp\in\mathrm{Pl}(L)$, let
	\begin{equation}\label{4.14..}
		\mathcal{H}_\wp={C}^\infty_c\Big(\mathbb{G}(\mathcal{O}_\wp)\big\backslash \mathbb{G}(L_\wp)\big/ \mathbb{G}(\mathcal{O}_\wp)\Big),
	\end{equation}
	the convolution algebra of compactly supported $\mathbb{G}(\mathcal{O}_\wp)$-biinvariant functions on $\mathbb{G}(L_\wp)$. Then $\mathcal{H}_\wp$ acts on the right by convolution on 
	${C}^\infty\big(\mathbb{G}(\mathbb{A}_L),{C}^\infty(\mathbb{S}^3)\big)$. This action preserves the subspace
	\begin{equation}
		C^\infty\big(\mathbb{G}(\mathbb{A}_L)/K_f,C^\infty(\mathbb{S}^3)\big)\hookrightarrow C^\infty\big(\mathbb{G}(\mathbb{A}_L),C^\infty(\mathbb{S}^3)\big),
	\end{equation}
	and commutes with the left $\mathbb{G}(L)$-action. Therefore, it passes to an action on ${C}^\infty(U(X),\pi^*\mathscr{F})$ by \eqref{4.5}.
	
	More precisely, consider the Bruhat-Tits tree
	\begin{equation}
		\mathfrak{X}_\wp=\mathrm{PGL}_2(L_\wp)/\mathrm{PGL}_2(\mathcal{O}_\wp)
	\end{equation}
	which is regular with degree $q_\wp+1$. Let $d_{\mathfrak{X}_\wp}(\cdot,\cdot)$ be the natural metric on $\mathfrak{X}_\wp$ such that the distance between nearest neighbors is $1$. For any $k\in\mathbb{N}$, we can define an operator $T_{\wp^{2k}}$ acting on ${C}^\infty(U(X),\pi^*\mathscr{F})$ by
	\begin{equation}\label{4.7}
		T_{\wp^{2k}}(s)(y)=\sum_{d_{\mathfrak{X}_\wp}(y,y')=2k}s(y'),
	\end{equation}
	where we regard $s\in{C}^\infty(U(X),\pi^*\mathscr{F})$ as an element of ${C}^\infty\big(\mathbb{G}(\mathbb{A}_L),{C}^\infty(\mathbb{S}^3)\big)$ through \eqref{4.5}.  The operator $T_{\wp^{2k}}$ also restricts to actions on subspaces listed in \eqref{4.12''}. Also, the cardinality of the summation set in \eqref{4.7} is given by
	\begin{equation}\label{4.7'}
		\lv\{y'\in \mathfrak{X}_\wp\mid d_{\mathfrak{X}_\wp}(y,y')=2k\}\rv=q_\wp^{2k-1}(q_\wp+1).
	\end{equation}
	Note that $\mathcal{H}_\wp$ is generated by $T_{\wp^2}$, rather than $T_\wp$ (which does not belong to $\mathcal{H}_\wp$). This is because there are two orbits of $\mathbb{G}(L_\wp)\cong \mathrm{SL}_2(L_\wp)$ among the vertices of $\mathfrak{X}_\wp$.

	\subsection{Hecke-Maass forms}\label{S5.4}

	The Hecke operators $\{T_{\wp^{2k}}\}_{\wp,k}$ commute with each other, and when restricting to $C^\infty(X,\mathscr{F})$, they commute with the Laplacian $\Delta^{\mathscr{F}}$ given in  \eqref{2.14}. Therefore, we can assume that the Maass form $u_{\lambda}\in C^\infty(X,\mathscr{F})$ in \eqref{2.36} is a \emph{joint} eigensection of $\Delta^{\mathscr{F}}$ and all $\{T_{\wp^{2k}}\}_{\wp,k}$,	called a \emph{Hecke-Maass} form. By the construction in \eqref{3.7} and \eqref{3.8}, the lift $s_{\lambda}\in C^\infty\big(U(X),\pi^*\mathscr{F}\big)$ of $u_\lambda\in C^\infty(X,\mathscr{F})$ is also a joint eigensection of all $\{T_{\wp^{2k}}\}_{\wp,k}$.

\section{Hecke Recurrence}\label{D}

	In this section, we prove the Hecke recurrence property stated in Proposition \ref{D1'}. In \cref{SH1}, we introduce the notion of Hecke recurrence at a prime $\wp$. In \cref{SH2}, we show that a Hecke eigensection on a tree cannot concentrate near a single vertex. In \cref{SH3}, we complete the proof of Hecke recurrence.

We work under the notation fixed in \cref{s1.6} and the setup \eqref{diag1'}, \eqref{diag2}, and \eqref{diag}.

	\subsection{$\mathbb{G}(L_\wp)$-recurrence}\label{SH1}

	Let us recall the definition of the Hecke recurrence property for a measure, see for instance \cite[Definition 2.3]{MR2195133}.

	\begin{defi}
		For a \emph{finite prime} $\wp$ and a finite measure $\mu_{U(X)}$ on $U(X)$, we say that $\mu_{U(X)}$ is $\mathbb{G}(L_\wp)$-\emph{recurrent} if
		for any Borel set $U\subset U(X)$ and $\mu_{U(X)}$-almost every $y=[(g_\infty,1)]\in U$, the set $\{g_\wp\in \mathbb{G}(L_\wp)\mid [(g_\infty,g_\wp)]\in U\}$ is unbounded. Equivalently, for $\mu_{U(X)}$-almost every $y\in U$, its Bruhat-Tits tree $\mathfrak{X}_\wp(y)$ intersects $U$ at infinitely many points.
	\end{defi}

Our goal in this section is to prove the following recurrent property, a precise version of Proposition \ref{D1'}.
	\begin{prop}\label{D1}
		Let $\{s_{\lambda_i}\in {C}^\infty(U(X),\pi^*\mathscr{F})\}_{i\in\mathbb{N}}$ be a sequence lifts of Hecke-Maass forms defined in \cref{S5.4}. Then for any weak star limit $\mu_{\pi^*\mathscr{M}}$ of $\{\lv s_{\lambda_i}\rv^2_{\mathbb{C}}dv_{\pi^*\mathscr{M}}\}_{i\in\mathbb{N}}$, its projection $\mu_{U(X)}$ to $U(X)$ is $\mathbb{G}(L_\wp)$-recurrent.
	\end{prop}
	
	\subsection{Nonconcentration}\label{SH2}
First, we prove that the restriction of a Hecke eigensection to a Hecke tree does not concentrate near any vertex.
	\begin{prop}
		There is $C>0$ such that for any $n\in\mathbb{N}$ and eigensection $s\in{C}^\infty(U(X),\pi^*\mathscr{F})$ of $T_{\wp^2}$, we have for $\lv s\rv_{\pi^*\mathscr{F}}^2\in{C}^\infty(U(X))$,
		\begin{equation}\label{5.1}
			\sum_{k=0}^nT_{\wp^{2k}}\big(\lv s\rv_{\pi^*\mathscr{F}}^2\big)(y)\geqslant Cn\lv s(y)\rv_{\pi^*\mathscr{F}}^2.
		\end{equation}
	\end{prop}

	\begin{pro}
		By \eqref{4.7} and Cauchy-Schwarz inequality, we have
		\begin{equation}\label{5.2}
			\begin{split}
				\lv T_{\wp^{2k}}(s)(y)\rv_{\pi^*\mathscr{F}}^2&\leqslant \Big(\sum_{d_{\mathfrak{X}_\wp}(y,y')=2k}\lv s(y')\rv_{\pi^*\mathscr{F}}\Big)^2\\
				&\leqslant\lv\{y'\in T_\wp(y)\mid d_{\mathfrak{X}_\wp}(y,y')=2k\}\rv\cdot T_{\wp^{2k}}\big(\lv s\rv_{\pi^*\mathscr{F}}^2\big)(y)\\
				&\leqslant Cq_\wp^{2k}T_{\wp^{2k}}\big(\lv s\rv_{\pi^*\mathscr{F}}^2\big)(y),
			\end{split}
		\end{equation}
		and similarly, 
		\begin{equation}\label{5.3}
			\begin{split}
				\bbv \sum_{k=0}^{n}T_{\wp^{2k}}(s)(y)\bbv_{\pi^*\mathscr{F}}^2&\leqslant \Big(\sum_{k=0}^{n}\sum_{d_{\mathfrak{X}_\wp}(y,y')=2k}\lv s(y')\rv_{\pi^*\mathscr{F}}\Big)^2\\
				&\leqslant\Bv\bigcup_{k\leqslant n}\{y'\in T_\wp(y)\mid d_{\mathfrak{X}_\wp}(y,y')=2k\}\Bv\cdot\sum_{k=0}^{n}T_{\wp^{2k}}\big(\lv s\rv_{\pi^*\mathscr{F}}^2\big)(y)\\
				&=\Big(1+q_\wp(q_\wp+1)+\cdots +q_\wp^{2k-1}(q_\wp+1)\Big)\sum_{k=0}^{n}T_{\wp^{2k}}\big(\lv s\rv_{\pi^*\mathscr{F}}^2\big)(y)\\
				&\leqslant Cq_\wp^{2n} \sum_{k=0}^{n}T_{\wp^{2k}}\big(\lv s\rv_{\pi^*\mathscr{F}}^2\big)(y).
			\end{split}
		\end{equation}
		
		We denote the eigenvalues of $T_{\wp^{2k}}$ by $\lambda_{\wp^{2k}}$, that is,
		\begin{equation}\label{5.4'}
			T_{\wp^{2k}}s=\lambda_{\wp^{2k}} s.
		\end{equation}
		Then define $\lambda_{\wp}$ by
		\begin{equation}
			\lambda_{\wp}^2=\lambda_{\wp^2}+(q_\wp+1),
		\end{equation} 
		so that $\lambda_{\wp}$ can be viewed as an eigenvalue of $T_{\wp}$. Following \cite[Lemma 8.3]{MR2195133}, we consider two cases, $\lv\lambda_{\wp}\rv>2\sqrt{q_\wp}$ and $\lv\lambda_{\wp}\rv\leqslant2\sqrt{q_\wp}$.

		The first case is straightforward. Using Chebyshev polynomials, we can show by induction that
		\begin{equation}\label{5.6'}
			\sum_{k=0}^{n}\lambda_{\wp^{2k}}=q_\wp^n\frac{\sinh (2n+1)\theta}{\sinh\theta},
		\end{equation}
		where $\theta=\cosh^{-1}\big(\frac{\lambda_{\wp}}{2\sqrt{q_\wp}}\big)$. Clearly
		\begin{equation}\label{5.7}
			\frac{\sinh (2n+1)\theta}{\sinh\theta}=\big(e^{2n\theta}+e^{-2n\theta}\big)+\big(e^{2(n-1)\theta}+e^{-2(n-1)\theta}\big)+\cdots+1\geqslant 2n+1.
		\end{equation}
		By \eqref{5.3}, \eqref{5.4'}, \eqref{5.6'}, and \eqref{5.7}, we get
		\begin{equation}\label{6.8}
			\begin{split}
				\sum_{k=0}^{n}T_{\wp^{2k}}\big(\lv s\rv_{\pi^*\mathscr{F}}^2\big)(y)&\geqslant Cq_\wp^{-2n}\bbv \sum_{k=0}^{n}T_{\wp^{2k}}(s)(y)\bbv_{\pi^*\mathscr{F}}^2\\
				&=Cq_\wp^{-2n}\bbv\sum_{k=0}^{n}\lambda_{\wp^{2k}}\bbv^2\lv s(y)\rv_{\pi^*\mathscr{F}}^2\geqslant Cn^2\lv s(y)\rv_{\pi^*\mathscr{F}}^2,
			\end{split}
		\end{equation}
		which is stronger than \eqref{5.1}.
		
		For the second case, we have
		\begin{equation}\label{6.9}
			\sum_{k=0}^{n}\lambda_{\wp^{2k}}=q_\wp^n\frac{\sin (2n+1)\theta}{\sin\theta},
		\end{equation}
		where $\theta=\cos^{-1}\big(\frac{\lambda_{\wp}}{2\sqrt{q_\wp}}\big)$. Moreover, by \eqref{5.2}, \eqref{5.4'}, and \eqref{6.9}, we obtain
		\begin{equation}\label{6.10}
			\begin{split}
				T_{\wp^{2k}}\big(\lv s\rv_{\pi^*\mathscr{F}}^2\big)(y)	&\geqslant Cq_\wp^{-2k}\lv T_{\wp^{2k}}(s)(y)\rv_{\pi^*\mathscr{F}}^2=Cq_\wp^{-2k}\lambda_{\wp^{2k}}^2\lv s(y)\rv_{\pi^*\mathscr{F}}^2\\
				&=C\Big(\frac{\sin (2k+1)\theta-q_\wp^{-1}\sin (2k-1)\theta}{\sin\theta}\Big)^2\lv s(y)\rv_{\pi^*\mathscr{F}}^2\\
				&\geqslant C\big(\sin (2k+1)\theta-q_\wp^{-1}\sin (2k-1)\theta\big)^2\lv s(y)\rv_{\pi^*\mathscr{F}}^2.
			\end{split}
		\end{equation}

		We shall deduce \eqref{5.1} by combining \eqref{6.9} and \eqref{6.10}. Without loss of generality, we suppose that $0\leqslant\theta\leqslant\pi/2$. By \eqref{6.8} and \eqref{6.9}, if $n\theta\leqslant\pi/3<\pi/2$, then $\sin(2n+1)\theta/\sin\theta\geqslant Cn$ and hence
		\begin{equation}\label{6.11'}
			\begin{split}
				\sum_{k=0}^{n}T_{\wp^{2k}}\big(\lv s\rv_{\pi^*\mathscr{F}}^2\big)(y)\geqslant C\Big(\frac{\sin (2n+1)\theta}{\sin\theta}\Big)^2\lv s(y)\rv_{\pi^*\mathscr{F}}^2\geqslant Cn^2\lv s(y)\rv_{\pi^*\mathscr{F}}^2,
			\end{split}
		\end{equation}
		which is again stronger than \eqref{5.1}. We now claim that if $n\theta\geqslant\pi$, then
		\begin{equation}\label{6.12}
			\sum_{k=0}^{n}\mathbbm{1}_{[\frac{\pi}{5},\frac{4\pi}{5}]}\big((2k+1)\theta\mod{\pi}\big)\geqslant Cn,
		\end{equation}
		namely that in the sequence $(2k+1)\theta\mod{\pi}$ for $k=0,\cdots, n$, at least $Cn$ of them belong to $[\pi/5,4\pi/5]$. Assuming \eqref{6.12}, it follows from \eqref{6.10} that if $n\theta\geqslant\pi$,
		\begin{equation}\label{6.13}
			\begin{split}
				\sum_{k=0}^{n}T_{\wp^{2k}}\big(\lv s\rv_{\pi^*\mathscr{F}}^2\big)(y)\geqslant C\Big(\sin\frac{\pi}{5}-\frac{1}{2}\Big)^2n\lv s(y)\rv_{\pi^*\mathscr{F}}^2,
			\end{split}
		\end{equation}
		since $q_\wp\geqslant 2$. If instead $\pi/3<n\theta<\pi$, then $n\theta/3<\pi/3$, and by \eqref{6.11'} we have
		\begin{equation}\label{6.14}
			\sum_{k=0}^{n}T_{\wp^{2k}}\big(\lv s\rv_{\pi^*\mathscr{F}}^2\big)(y)\geqslant \sum_{k=0}^{n/3}T_{\wp^{2k}}\big(\lv s\rv_{\pi^*\mathscr{F}}^2\big)(y)\geqslant \frac{Cn}{3}\lv s(y)\rv_{\pi^*\mathscr{F}}^2.
		\end{equation}
		Combining \eqref{6.11'}, \eqref{6.13}, and \eqref{6.14}, we obtain \eqref{5.1}.

		It remains to verify \eqref{6.12}. Identify the interval $[0,\pi]$ with a circle by gluing its endpoints, so that increasing the angle corresponds to a clockwise rotation. In the sequence $(2k+1)\theta\mod{\pi}$, each step we move clockwise by $2\theta$ if $2\theta\leqslant \pi/2$, or equivalently we move counterclockwise by $2(\pi/2-\theta)\leqslant\pi/2$ if $2\theta\geqslant \pi/2$. The key feature of the interval $[\pi/5,4\pi/5]$ is that its length is strictly greater than $\pi/2$, so a single step of size at most $\pi/2$, either clockwise or counterclockwise, cannot jump over it. To see necessity of this, take $\theta=\pi/4$, then $(2k+1)\theta\mod{\pi}\subseteq\{\pi/4,3\pi/4\}$, and the sequence never intersects $[\pi/4+\varepsilon,3\pi/4-\varepsilon]$ for any small $\varepsilon>0$, so
		\begin{equation}
			\sum_{k=0}^{n}\mathbbm{1}_{[\frac{\pi}{4}+\varepsilon,\frac{3\pi}{4}-\varepsilon]}\big((2k+1)\theta\mod{\pi}\big)=0.
		\end{equation}
		
		We now prove \eqref{6.12}. We consider two cases, $\theta\leqslant \pi/10$ and $\theta\geqslant \pi/10$.
		
		For the first case,  our step size is at most $2\theta\leqslant \pi/5$. Consider $([\pi/2\theta]+1)$ consecutive steps, and we enter $[\pi/5,4\pi/5]$ at the $i$-th step and leave at $j$-th step, then $[\pi/5,4\pi/5]\subseteq[(i-1)\text{-th step},j\text{-th step}]$. Counting the length, we get
		\begin{equation}
			\Big(\frac{4\pi}{5}-\frac{\pi}{5}\Big)<2\theta(j-i+1).
		\end{equation}
		Thus, among $([\pi/2\theta]+1)$ steps, at least $(j-i)>3\pi/(10\theta)>1$ (since $0\leqslant\theta\leqslant \pi/5$) lie in $[\pi/5,4\pi/5]$, and the proportion is greater than
		\begin{equation}
			\frac{3\pi}{10\theta}\Big/\Big(\frac{\pi}{2\theta}+1\Big)>C.
		\end{equation}
		Since $n\geqslant \pi/\theta>([\pi/2\theta]+1)$, a division with remainder yields \eqref{6.12}.
		
		For the second case, our step size is at least $2\theta\geqslant \pi/5$. Then among any four consecutive steps, at least one must lie in $[\pi/5,4\pi/5]$, otherwise, the total length of the complement of this interval would be exceeded. Again, a division with remainder leads to \eqref{6.12}.\qed
	\end{pro}

	\subsection{Proof of Hecke recurrence}\label{SH3}
	
	We now prove Proposition \ref{D1}. Let $\mu_{\pi^*\mathscr{M}}$ be a weak star limit of $\{\lv s_{\lambda_i}\rv^2_{\mathbb{C}}dv_{\pi^*\mathscr{M}}\}_{i\in\mathbb{N}}$ and $\mu_{U(X)}$ its projection to $U(X)$, equivalently, $\mu_{U(X)}$ is the corresponding weak star limit of $\{\lv s_{\lambda_i}\rv_{\pi^*\mathscr{F}}^2dv_{U(X)}\}$. For any Borel set $U\subseteq U(X)$, we apply \eqref{5.1} to $s_{\lambda_i}$ and pass to the weak star limit. This yields
	\begin{equation}\label{5.4}
		\int_{U(X)}	\sum_{k=0}^nT_{\wp^{2k}}(\mathbbm{1}_U)d\mu_{U(X)}\geqslant Cn\mu_{U(X)}(U).
	\end{equation}
	Intuitively, this inequality reflects the fact that, on the sphere of radius $2n$ in the Hecke tree $\mathfrak{X}_\wp(y)$, at least $Cn$ points are lying in $U$, see the proof of \cite[Theorem 8.1]{MR2195133}. For our purposes, a much weaker statement suffices. Fix $j\in\mathbb{N}$, let us form
	\begin{equation}
		U_j=\Big\{y\in U\mid \lv \mathfrak{X}_\wp(y)\cap U\rv\leqslant j\Big\}.
	\end{equation}
	For any $y\in U_j$, we have
	\begin{equation}
		\sum_{k=0}^nT_{\wp^{2k}}(\mathbbm{1}_{U_j})(y)=\lv \{y'\in \mathfrak{X}_\wp(y)\mid d_{\mathfrak{X}_\wp(y)}(y,y')\leqslant 2n\}\cap U_j\rv\leqslant j.
	\end{equation}
	Applying \eqref{5.4} to $U_j$, we get
	\begin{equation}
		\mu_{U(X)}(U_j)\leqslant \frac{C\int_{U(X)}	jd\mu_{U(X)}}{n}=\frac{Cj}{n}.
	\end{equation}
	Letting $n\to\infty$, we see that $\mu_{U(X)}(U_j)=0$. Consequently, $\mu_{U(X)}(\cup_{j=1}^\infty U_j)=0$. Since $\cup_{j=1}^\infty U_j$ is exactly the set of points in $U$ without the recurrent property, this completes the proof.

\section{Positive Entropy}\label{S7}
	
In this section, we prove the positive entropy property stated in Proposition \ref{E5''}. In \cref{S7.1}, we recall a criterion for positive entropy. In \cref{S7.2}, we prove an ampleness property for Hecke eigensections. In \cref{S7.3}, we complete the proof of positive entropy.

We retain the notation fixed in \cref{s1.6} and the setup \eqref{diag1'}, \eqref{diag2}, and \eqref{diag}. We also adopt the notation from \cref{SCb}, for example, we write $(g_\infty,g_f)\in \mathbb{G}(\mathbb{A}_L)$.

\subsection{Positive entropy for almost every ergodic component}\label{S7.1}

	Recall the geodesic flow $(g_t)$ on $U(X)$ defined in \eqref{2.12}. Instead of discussing the general definition of entropy, we present an equivalent formulation suited to our purposes, describing when a $(g_t)$-invariant measure $\mu_{U(X)}$ on $U(X)$ has positive entropy on almost every ergodic component.

		We form two one-parameter subgroups of $\mathrm{SL}_2(\mathbb{R})$
	\begin{equation}
		u^+_t=\left(\begin{matrix}1&0\\t&1\end{matrix}\right),\quad u^-_t=\left(\begin{matrix}1&t\\0&1\end{matrix}\right).
	\end{equation}
	For a fixed small $\delta>0$ and the geodesic flow $(g_t)$ defined in \eqref{2.12}, we set a neighborhood $U_{\delta,\varepsilon}$ of $\mathrm{Id}\in\mathrm{SL}_2(\mathbb{R})$ by
	\begin{equation}\label{7.2}
		U_{\delta,\varepsilon}=g_{(-\delta,\delta)}\cdot u^-_{(-\varepsilon,\varepsilon)}\cdot u^+_{(-\varepsilon,\varepsilon)}.
	\end{equation}
	Geometrically, $U_{\delta,\varepsilon}$ can be viewed as the $\varepsilon$-neighborhood of a segment of the diagonal subgroup.

	\begin{defi}
We say that a $(g_t)$-invariant measure $\mu_{U(X)}$ on $U(X)$ has positive entropy on almost every ergodic component if for almost every $[g_\infty]\in U(X)$, there exist $C,h>0$ such that for $\varepsilon$ small enough,
\begin{equation}\label{7.3}
\mu_{U(X)}\big([g_\infty U_{\delta,\varepsilon}]\big)\leqslant C\varepsilon^h.
\end{equation}
	\end{defi}

We refer to \cite[\S\,4.3.1, \S\,4.4.2]{Arizona} for the definition of entropy and for precise statements showing that \eqref{7.3} implies positive entropy on almost every ergodic component. The underlying idea is that, for a partition of $U(X)$ with diameter less than $\delta$, the atoms in its refinement under the time-one map $(g_1)$ are well approximated by $U_{\delta,\varepsilon}$.

The aim of this section is to prove the following bound on the mass of tubes, which provides a precise formulation of Proposition \ref{E5''}.
	
	\begin{prop}\label{E4}
		There are $C,h>0$ such that for any $[g_\infty]\in U(X)$, $\varepsilon$ small enough and  a lift $s\in {C}^\infty(U(X),\pi^*\mathscr{F})$ of Hecke-Maass form defined in \cref{S5.4}, we have for $\lv s\rv_{\pi^*\mathscr{F}}^2\in{C}^\infty(U(X))$ and $\varepsilon$ small enough,
		\begin{equation}\label{5.12}
			\int_{[g_\infty U_{\delta,\varepsilon}]}\lv s\rv_{\pi^*\mathscr{F}}^2dv_{X}\leqslant C\varepsilon^h.
		\end{equation}
	\end{prop}

\subsection{Hecke amplifier}\label{S7.2}

	By \eqref{4.7}, we have
	\begin{equation}
		T_{\wp^2}^2=T_{\wp^4}+(q_\wp-1)T_{\wp^2}+q_\wp(q_\wp+1)\mathrm{Id},
	\end{equation}
	which, together with \eqref{4.7'}, implies the following ampleness property.

	\begin{prop}\label{p5.3}
		There is $C>0$ such that for any Hecke eigensection $s\in{C}^\infty(U(X),\pi^*\mathscr{F})$, if we denote the corresponding eigenvalue of $T_{\wp^2}$ and $T_{\wp^4}$ with $\lambda_{\wp^2}$ and $\lambda_{\wp^4}$ respectively, then we have either $\lv\lambda_{\wp^2}\rv\geqslant Cq_\wp$ or $\lv\lambda_{\wp^4}\rv\geqslant Cq_\wp^2$.
	\end{prop}

	For any $Q>0$ we put
	\begin{equation}\label{5.6}
		\mathscr{P}_Q=\{\wp\in\mathrm{Pl}(L)\mid Q/2\leqslant q_\wp\leqslant Q\}.
	\end{equation}
	By Proposition \ref{p5.3}, we can let $\ell$ be either $1$ or $2$ and choose a subset $\mathscr{P}_Q'\subseteq \mathscr{P}_Q$ with 
	\begin{equation}\label{7.7}
\lv\mathscr{P}_Q'\rv\geqslant \frac{1}{2}\lv \mathscr{P}_Q\rv,\quad\lv\lambda_{\wp^{2\ell}}\rv\geqslant Cq_\wp^\ell \text{ for any } \wp\in \mathscr{P}_Q'.
	\end{equation}

Following \cite[Lemma 5.2]{MR4033919} and \cite[Proposition 30]{MR5036694}, for a subset $\mathscr{P}_Q''\subseteq\mathscr{P}_Q'$ to be specified soon, we define a \emph{global amplifier} $T_{Q}$ by
	\begin{equation}\label{7.8}
		T_{Q}=\sum_{\wp\in\mathscr{P}_Q''}T_{\wp^{2\ell}}.
	\end{equation}
	 The Hecke eigensection $s$ in Proposition \ref{p5.3} is also an eigensection of $T_Q$, and by \eqref{7.7} we have the following estimate for the associated eigenvalue,
	\begin{equation}\label{5.8}
	T_{Q}s=\lambda_{Q}s,\quad\lambda_{Q}=\sum_{\wp\in\mathscr{P}_Q''}\lambda_{\wp^{2\ell}}\geqslant C\lv\mathscr{P}_Q''\rv Q^\ell.
	\end{equation}
Let us denote $\mathrm{supp}(T_{Q})$ the support of $T_{Q}$ when we view it as a convolution kernel in \eqref{4.14..}. Using \eqref{4.7'}, the size of $\mathrm{supp}(T_{Q})$ satisfies
	\begin{equation}\label{5.9}	\bv\mathrm{supp}(T_{Q})\bv=\sum_{\wp\in\mathscr{P}_Q''}q_\wp^{2\ell-1}(q_\wp+1)\leqslant C\lv\mathscr{P}_Q''\rv Q^{2\ell}.
	\end{equation}

As in \cite[Lemma 3.3]{MR4033919}, applying the triangle inequality for \eqref{4.7} and \eqref{7.8}, we have
		\begin{equation}\label{5.13}
			\begin{split}
			\lambda_{Q}^2\int_{[g_\infty U_{\delta,\varepsilon}]}\bv s\bv_{\pi^*\mathscr{F}}^2dv_{U(X)}&=	\int_{[g_\infty U_{\delta,\varepsilon}]}\bv T_Q s\bv_{\pi^*\mathscr{F}}^2dv_{U(X)}\\
				&\leqslant \bigg({\sum_{g_f\in\mathrm{supp}(T_Q)}}\Big(\int_{[(g_\infty U_{\delta,\varepsilon}, g_f)]}\bv s\bv_{\pi^*\mathscr{F}}^2dv_{U(X)}\Big)^{1/2}\bigg)^2.
			\end{split}
		\end{equation}

To further estimate \eqref{5.13}, we recall a result from \cite[Lemma 3.4]{MR4033919}.
\begin{lemma}\label{l7.4}
There is $C>0$ such that for any $T_Q$ defined in \eqref{7.8}, we have for $\varepsilon$ small enough,
	\begin{equation}\label{7.12}
	\begin{split}
		\bigg({\sum_{g_f\in\mathrm{supp}(T_Q)}}&\Big(\int_{[(g_\infty U_{\delta,\varepsilon},g_f)]}\bv s\bv_{\pi^*\mathscr{F}}^2dv_{U(X)}\Big)^{1/2}\bigg)^2\\
		&\leqslant C \lv\Big\{g_f,g_f'\in \mathrm{supp}(T_Q)\mid [(g_\infty U_{3\delta,3\varepsilon},g_f)]\cap [g_\infty U_{3\delta,3\varepsilon},g_f']\neq \emptyset\Big\} \rv.
	\end{split}
\end{equation}
\end{lemma}

	\begin{pro}

		A key ingredient in the proof is the fact that when $\varepsilon$ is sufficiently small, the neighborhood $U_{\delta,\varepsilon}$ of the identity $\mathrm{Id}\in \mathrm{SL}_2(\mathbb R)$ is sufficiently flat that the group law on $\mathrm{SL}_2(\mathbb R)$ is well approximated by addition in the Lie algebra.	
		
We claim that there is $C>0$ such that for any small $\varepsilon>0$, there is a finite covering
\begin{equation}\label{7.15'}
	U(X)=\big\{[ g_{\infty,i}U_{\delta,\varepsilon}]\big\}_i
\end{equation}
of $U(X)$ so that any set $[g_{\infty}U_{\delta,\varepsilon}]$ intersects at most $C$ of these sets. To establish this, let us choose a maximal set of $\{[g_{\infty,i}]\}_i$ such that $\{[g_{\infty,i}U_{{\delta}/{3},{\varepsilon}/3}]\}_i$ are disadjoint. Then for any $[g_{\infty}]\in U(X)$, we have $[g_{\infty}U_{{\delta}/{3},{\varepsilon}/3}]$ intersects one of $[g_{\infty,i}U_{{\delta}/{3},{\varepsilon}/3}]$, hence
\begin{equation}
	[g_{\infty}]\in [g_{\infty,i}U_{{\delta}/{3},{\varepsilon}/3}U_{{\delta}/{3},{\varepsilon}/3}^{-1}]\subseteq [g_{\infty,i}U_{{\delta},{\varepsilon}}],
\end{equation}
and this implies \eqref{7.15'}. For $[g_{\infty}]\in U(X)$, if $[g_{\infty}U_{{\delta},{\varepsilon}}]\cap[g_{\infty,i}U_{{\delta},{\varepsilon}}]\neq\emptyset$, we can choose $g'_{\infty,i}\in U_{{\delta},{\varepsilon}}U_{{\delta},{\varepsilon}}^{-1}$ so that $[g_{\infty,i}]=[g_\infty g'_{\infty,i}]$, therefore, the sets $[g_{\infty,i}U_{{\delta}/{3},{\varepsilon}/3}]=[g_\infty g'_{\infty,i}U_{{\delta}/{3},{\varepsilon}/3}]$ are disadjoint and all belong to $[g_\infty U_{{\delta},{\varepsilon}}U_{{\delta},{\varepsilon}}^{-1}U_{{\delta}/{3},{\varepsilon}/3}]$. So the number is bounded by
\begin{equation}
	\frac{\mathrm{Vol}(U_{{\delta},{\varepsilon}}U_{{\delta},{\varepsilon}}^{-1}U_{{\delta}/{3},{\varepsilon}/3})}{\mathrm{Vol}(U_{{\delta}/{3},{\varepsilon}/3})}\sim\frac{(\delta+\delta+\delta/3)(\varepsilon+\varepsilon+\varepsilon/3)^2}{(\delta/3)(\varepsilon/3)^2}\leqslant C.
\end{equation}

By \eqref{5.11}, we have $[(g_\infty U_{\delta,\varepsilon},g_f)]=[g_L^{-1}g_\infty U_{\delta,\varepsilon}]$, which is of the same form as in the preceding discussion. Then by \eqref{7.15'}, we have
\begin{equation}
	\begin{split}
		\int_{[(g_\infty U_{\delta,\varepsilon},g_f)]}\bv s\bv_{\pi^*\mathscr{F}}^2dv_{U(X)}\leqslant \sum_{[g_{\infty,i} U_{\delta,\varepsilon}]\cap [(g_\infty U_{\delta,\varepsilon},g_f)]\neq\emptyset}\int_{[\Gamma g_{\infty,i} U_{\delta,\varepsilon}]}\bv s\bv_{\pi^*\mathscr{F}}^2dv_{U(X)},
	\end{split}
\end{equation}
which clearly implies
\begin{equation}\label{7.19}
	\begin{split}
		\Big(\int_{[(g_\infty U_{\delta,\varepsilon},g_f)]}&\bv s\bv_{\pi^*\mathscr{F}}^2dv_{U(X)}\Big)^{1/2}\\
		&\leqslant \sum_{[g_{\infty,i} U_{\delta,\varepsilon}]\cap [(g_\infty U_{\delta,\varepsilon},g_f)]\neq\emptyset}\Big(\int_{[\Gamma g_{\infty,i} U_{\delta,\varepsilon}]}\bv s\bv_{\pi^*\mathscr{F}}^2dv_{U(X)}\Big)^{1/2}.
	\end{split}
\end{equation}

From \eqref{7.19}, the left hand side of \eqref{7.12} can be bounded by
\begin{equation}\label{7.20}
	\begin{split}
		\bigg({\sum_{g_f\in\mathrm{supp}(T_Q)}}&\Big(\int_{[(g_\infty U_{\delta,\varepsilon},g_f)]}\bv s\bv_{\pi^*\mathscr{F}}^2dv_{U(X)}\Big)^{1/2}\bigg)^2\\
		\leqslant& \bigg(\sum_{\substack{g_f\in\mathrm{supp}(T_Q),\\ [g_{\infty,i} U_{\delta,\varepsilon}]\cap [(g_\infty U_{\delta,\varepsilon},g_f)]\neq\emptyset}} \Big(\int_{[g_{\infty,i} U_{\delta,\varepsilon}]}\bv s\bv_{\pi^*\mathscr{F}}^2dv_{U(X)}\Big)^{1/2}\bigg)^2\\
		\leqslant& \bigg(\sum_i\int_{[g_{\infty,i} U_{\delta,\varepsilon}]}\bv s\bv_{\pi^*\mathscr{F}}^2dv_{U(X)}\bigg)\bigg(\sum_{i}\Big(\sum_{[g_{\infty,i} U_{\delta,\varepsilon}]\cap [(g_\infty U_{\delta,\varepsilon},g_f)]\neq\emptyset}1\Big)^2\bigg).
	\end{split}
\end{equation}
By the construction of $\{[g_{\infty,i} U_{\delta,\varepsilon}]\}$, each point $[g_\infty]$ belongs to at most $C$ of these sets, hence
\begin{equation}\label{7.21}
	\sum_i\int_{[g_{\infty,i} U_{\delta,\varepsilon}]}\bv s\bv_{\pi^*\mathscr{F}}^2dv_{U(X)}\leqslant C\int_{U(X)}\bv s\bv_{\pi^*\mathscr{F}}^2dv_{U(X)}\leqslant C.
\end{equation}
On the other hand, we have
\begin{equation}\label{7.22}
	\begin{split}
		&\sum_{i}\bigg(\sum_{[g_{\infty,i} U_{\delta,\varepsilon}]\cap [(g_\infty U_{\delta,\varepsilon},g_f)]\neq\emptyset}1\bigg)^2\\
		&=\lv\Big\{i,g_f,g_f'\mid [g_{\infty,i} U_{\delta,\varepsilon}]\cap [(g_\infty U_{\delta,\varepsilon},g_f)]\neq\emptyset,\  [g_{\infty,i} U_{\delta,\varepsilon}]\cap[(g_\infty U_{\delta,\varepsilon},g_f')]\neq\emptyset\Big\}\rv\\
		&\leqslant C \lv\Big\{g_f,g_f'\in \mathrm{supp}(T_Q)\mid [(g_\infty U_{3\delta,3\varepsilon},g_f)]\cap [g_\infty U_{3\delta,3\varepsilon},g_f']\neq \emptyset\Big\} \rv.
	\end{split}
\end{equation}
Here, the last inequality holds because, for each fixed pair $(g_f,g_f')$, there are at most $C$ indices $i$ satisfying the two intersection conditions, moreover, we have $g_{\infty,i}\in[(g_\infty U_{\delta,\varepsilon} U_{\delta,\varepsilon}^{-1},g_f)]\cap [(g_\infty U_{\delta,\varepsilon} U_{\delta,\varepsilon}^{-1},g_f')]\neq\emptyset$ and $U_{\delta,\varepsilon} U_{\delta,\varepsilon}^{-1}\subseteq U_{3\delta,3\varepsilon}$.

Putting together \eqref{7.20}, \eqref{7.21}, and \eqref{7.22}, we obtain \eqref{7.12}.\qed
	\end{pro}

The following result is taken from \cite[Lemma 5.1]{MR4033919}.
\begin{lemma}\label{l7.5}
There exists a subset of good primes $\mathscr{P}_Q''\subseteq \mathscr{P}_Q'$ with
\begin{equation}\label{7.13}
	\lv\mathscr{P}_Q''\rv\geqslant \lv\mathscr{P}_Q'\rv-C\log(Q).
\end{equation}
Moreover, there exist $C,c>0$ such that for $T_Q$ defined in \eqref{7.8}, we have
\begin{equation}\label{7.14}
	\begin{split}
		\Bv\Big\{g_f,g_f'\in \mathrm{supp}(T_{Q})\mid [(g_\infty U_{3\delta,3\varepsilon},g_f)]\cap[(g_\infty U_{3\delta,3\varepsilon},g_f')]\neq \emptyset\Big\}\Bv&\\
		\leqslant C\big(\lv \mathscr{P}_Q''\rv^2+\bv\mathrm{supp}(T_Q)\bv\big)&\text{ if }Q\leqslant C\varepsilon^{-c}.
	\end{split}
\end{equation}
	\end{lemma}
	
\begin{pro}
This is the arithmetic input in the positive entropy argument. A complete proof can be found in the original paper,  based on a commutator argument in \cite[Lemma 3.3]{MR1957735}. As noted in \cite[\S\,3.1.2]{MR5036694}, one modification required in the number field setting is that the Diophantine result \cite[Proposition 4.7]{MR4033919} should be replaced by its number field generalization \cite[Proposition 10]{MR5036694}. We briefly indicate the idea of this generalization.

Recall $\mathbb{G}$ defined in \eqref{5.5.}. For any $g_L\in\mathbb{G}(L)$, let $g_L$ act on $D_{a,b}(L)$ by left multiplication. With respect to the basis $\{1,\mathrm{i},\mathrm{j},\mathrm{ij}\}$, we can view $g_L\in \mathrm{SL}_4(L)$. For any matrix entry $(g_L)_{ij}$, the product formula gives
\begin{equation}\label{7.23}
\lv(g_L)_{ij}\rv_\mathbb{C}\prod_{\text{infinite nonidentity\ }\tau\in\mathrm{Pl}(L)}\lv\tau(g_L)_{ij}\rv_\mathbb{C}\prod_{\text{finite\ }\wp\in\mathrm{Pl}(L)}\lv(g_L)_{ij}\rv_{\wp}=1
\end{equation}
whenever $(g_L)_{ij}\neq 0$. Thus, if certain conditions imply that the left hand side of \eqref{7.23} is sufficiently small, we must have $(g_L)_{ij}=0$.

Compared with the group over $\mathbb{Q}$ case in \cite[Lemma 3.3]{MR1957735} and \cite[Lemma 5.1]{MR4033919}, in \eqref{7.23}, the factors indexed by nonidentity infinite places are new. The key observation of \cite[(3.1)]{MR5036694} is that these extra factors do not affect the argument, because $\lv\tau(g_L)_{ij}\rv_\mathbb{C}$ are bounded. Indeed, similar to the argument leading to \eqref{4.2'}, the element $\tau(g_L)$ acts on $D_{\tau(a),\tau(b)}(L)$ and can be viewed $\tau(g_L)\in \mathrm{O}_4(L)\cap \mathrm{SL}_4(L)=\mathrm{SO}_4(L)$. Hence all of its entries are bounded.\qed
\end{pro}

\subsection{Proof of positive entropy}\label{S7.3}

	By \eqref{5.6}, \eqref{7.7}, \eqref{7.13}, and the prime ideal theorem for number fields, we have for $Q$ large,
	\begin{equation}\label{7.15}
		\lv\mathscr{P}_Q''\rv\sim C\frac{Q}{\log Q}
	\end{equation}

	Combining \eqref{5.8}, \eqref{5.9}, \eqref{5.13}, \eqref{7.12}, \eqref{7.14}, and \eqref{7.15}, it follows that
	\begin{equation}\label{5.15}
		\begin{split}
			\int_{[g_\infty U_{\delta,\varepsilon}]}\lv s\rv_{\pi^*\mathscr{F}}^2dv_{U(X)}&\leqslant C\lv\lambda_{Q}\rv^{-2}\Big(\lv \mathscr{P}_Q''\rv^2+\bv\mathrm{supp}(T_Q)\bv\Big)\\
			&\leqslant C(Q^{-2\ell}+\lv \mathscr{P}_Q''\rv^{-1})\leqslant C\big(\varepsilon^{\ell c}-\varepsilon^{c}\log\varepsilon\big),
		\end{split}
	\end{equation}
	which implies \eqref{5.12}.

	\addcontentsline{toc}{section}{References}
	\bibliographystyle{abbrv}

\begin{thebibliography}{10}
	
	\bibitem{MR2838248}
	J.-M. Bismut, X.~Ma, and W.~Zhang.
	\newblock Op\'{e}rateurs de {T}oeplitz et torsion analytique asymptotique.
	\newblock {\em C. R. Math. Acad. Sci. Paris}, 349(17-18):977--981, 2011.
	
	\bibitem{MR3615411}
	J.-M. Bismut, X.~Ma, and W.~Zhang.
	\newblock Asymptotic torsion and {T}oeplitz operators.
	\newblock {\em J. Inst. Math. Jussieu}, 16(2):223--349, 2017.
	
	\bibitem{MR1957735}
	J.~Bourgain and E.~Lindenstrauss.
	\newblock Entropy of quantum limits.
	\newblock {\em Comm. Math. Phys.}, 233(1):153--171, 2003.
	
	\bibitem{MR3260861}
	S.~Brooks and E.~Lindenstrauss.
	\newblock Joint quasimodes, positive entropy, and quantum unique ergodicity.
	\newblock {\em Invent. Math.}, 198(1):219--259, 2014.
	
	\bibitem{cekić2024semiclassicalanalysisprincipalbundles}
	M.~Ceki\'{c} and T.~Lefeuvre.
	\newblock Semiclassical analysis on principal bundles.
	\newblock {\em arXiv: 2405.14846}, 2024.
	
	\bibitem{MR818831}
	Y.~Colin~de Verdi{\`e}re.
	\newblock Ergodicit\'{e} et fonctions propres du laplacien.
	\newblock {\em Comm. Math. Phys.}, 102(3):497--502, 1985.
	
	\bibitem{Arizona}
	M.~Einsiedler and T.~Wald.
	\newblock Quantum unique ergodicity on {$\Gamma\backslash\mathbb{H}$}.
	\newblock {\em Arizona Winter School, Tucson}, 2010.
	
	\bibitem{MR1859345}
	E.~Lindenstrauss.
	\newblock On quantum unique ergodicity for {$\Gamma\backslash\Bbb H\times\Bbb
		H$}.
	\newblock {\em Internat. Math. Res. Notices}, (17):913--933, 2001.
	
	\bibitem{MR2195133}
	E.~Lindenstrauss.
	\newblock Invariant measures and arithmetic quantum unique ergodicity.
	\newblock {\em Ann. of Math. (2)}, 163(1):165--219, 2006.
	
	\bibitem{Ma-Ma}
	M.~Ma and Q.~Ma.
	\newblock Semiclassical analysis, geometric representation and quantum
	ergodicity.
	\newblock {\em Communications in Mathematical Physics}, 405(11):259, 2024.
	
	\bibitem{MR4665497}
	Q.~Ma.
	\newblock Toeplitz operators and the full asymptotic torsion forms.
	\newblock {\em J. Funct. Anal.}, 286(3):Paper No. 110210, 74pp, 2024.
	
	\bibitem{MR2339952}
	X.~Ma and G.~Marinescu.
	\newblock {\em Holomorphic {M}orse inequalities and {B}ergman kernels}, volume
	254 of {\em Progress in Mathematics}.
	\newblock Birkh\"{a}user Verlag, Basel, 2007.
	
	\bibitem{MR3307755}
	D.~W. Morris.
	\newblock {\em Introduction to arithmetic groups}.
	\newblock Deductive Press, [place of publication not identified], 2015.
	
	\bibitem{MR4611826}
	M.~Puchol.
	\newblock The asymptotics of the holomorphic analytic torsion forms.
	\newblock {\em J. Lond. Math. Soc. (2)}, 108(1):80--140, 2023.
	
	\bibitem{MR1266075}
	Z.~Rudnick and P.~Sarnak.
	\newblock The behaviour of eigenstates of arithmetic hyperbolic manifolds.
	\newblock {\em Comm. Math. Phys.}, 161(1):195--213, 1994.
	
	\bibitem{MR5036694}
	Z.~Shem-Tov and L.~Silberman.
	\newblock Arithmetic quantum unique ergodicity for products of hyperbolic 2-
	and 3-spaces.
	\newblock {\em J. Anal. Math.}, 158(2):413--448, 2026.
	
	\bibitem{MR4033919}
	L.~Silberman and A.~Venkatesh.
	\newblock Entropy bounds and quantum unique ergodicity for {H}ecke
	eigenfunctions on division algebras.
	\newblock In {\em Probabilistic methods in geometry, topology and spectral
		theory}, volume 739 of {\em Contemp. Math.}, pages 171--197. Amer. Math.
	Soc., [Providence], RI, [2019] \copyright 2019.
	
	\bibitem{MR2680500}
	K.~Soundararajan.
	\newblock Quantum unique ergodicity for {${\rm SL}_2(\Bbb Z)\backslash\Bbb H$}.
	\newblock {\em Ann. of Math. (2)}, 172(2):1529--1538, 2010.
	
	\bibitem{MR3558990}
	M.~Viana and K.~Oliveira.
	\newblock {\em Foundations of ergodic theory}, volume 151 of {\em Cambridge
		Studies in Advanced Mathematics}.
	\newblock Cambridge University Press, Cambridge, 2016.
	
	\bibitem{MR4279905}
	J.~Voight.
	\newblock {\em Quaternion algebras}, volume 288 of {\em Graduate Texts in
		Mathematics}.
	\newblock Springer, Cham, [2021] \copyright 2021.
	
	\bibitem{MR0402834}
	A.~I. \v{S}nirel$'$man.
	\newblock Ergodic properties of eigenfunctions.
	\newblock {\em Uspehi Mat. Nauk}, 29(6(180)):181--182, 1974.
	
	\bibitem{MR1838659}
	S.~A. Wolpert.
	\newblock Semiclassical limits for the hyperbolic plane.
	\newblock {\em Duke Math. J.}, 108(3):449--509, 2001.
	
	\bibitem{MR916129}
	S.~Zelditch.
	\newblock Uniform distribution of eigenfunctions on compact hyperbolic
	surfaces.
	\newblock {\em Duke Math. J.}, 55(4):919--941, 1987.
	
	\bibitem{MR2952218}
	M.~Zworski.
	\newblock {\em Semiclassical analysis}, volume 138 of {\em Graduate Studies in
		Mathematics}.
	\newblock American Mathematical Society, Providence, RI, 2012.
	
\end{thebibliography}
	\def\cprime{$'$} \def\cprime{$'$}

\end{document}